\documentclass[12pt]{amsart}
\counterwithin{equation}{section}

\allowdisplaybreaks

\usepackage{bbm}
\usepackage{amssymb}
\usepackage{amsmath}
\usepackage[all]{xy}
\usepackage{tikz}
\usepackage{amsfonts}
\usepackage{mathrsfs}
\usepackage{latexsym}
\usepackage{graphicx}
\usepackage{enumerate}
\usepackage{amscd,amssymb,amsmath,amsbsy,amsthm}
	\definecolor{linkred}{rgb}{0.7,0.2,0.2}
	\definecolor{linkblue}{rgb}{0,0.2,0.6}
	\definecolor{linkgreen}{rgb}{0,0.6,0.2}
\usepackage[colorlinks,plainpages,
    linkcolor=linkblue,
    citecolor=linkgreen,
    urlcolor=linkred]{hyperref}
\usepackage[capitalise]{cleveref}
\crefname{equation}{}{}
\usepackage{stmaryrd}

\newtheorem{theorem}{Theorem}[section]
\newtheorem*{theorem*}{Theorem}
\newtheorem{definition}[theorem]{Definition}
\newtheorem{lemma}[theorem]{Lemma}
\newtheorem{corol}[theorem]{Corollary}
\newtheorem{prop}[theorem]{Proposition}

\newtheorem{example}[theorem]{Example}
\newtheorem{remark}[theorem]{Remark}

\def\calK{{\mathcal{K}}}
\def\calO{{\mathcal{O}}}
\def\calR{{\mathcal{R}}}
\def\bfN{{\mathbf{N}}}
\def\Gr{\operatorname{Gr}}
\def\bGr{\overline{\operatorname{Gr}}}
\def\pt{\mathsf{pt}}
\def\gr{\operatorname{gr}}
\def\Eu{\operatorname{Eu}}
\def\Diff{\operatorname{\mathsf{Diff}}}
\def\Hom{\operatorname{Hom}}
\def\GL{\operatorname{GL}}

\tikzset{anchorbase/.style={>=To,baseline={([yshift=-0.5ex]current bounding box.center)}}}

\begin{document}

\title[Shifted twisted Yangians and Coulomb branch]{Quivers with Involutions and Shifted Twisted Yangians via Coulomb Branches II}

\author{Zichang Wang}
\address{Qiuzhen College, Tsinghua University, Beijing, China, 100084}
\email{w-zc21@mails.tsinghua.edu.cn}

\begin{abstract}
    To a quiver with involution, we show that there is an algebra homomorphism from the corresponding shifted twisted Yangian to the quantized Coulomb branch algebra of the 3d $\mathcal{N} = 4$ involution-fixed part of the quiver gauge theory in the second symmetric power case.
\end{abstract}

\maketitle

\setcounter{tocdepth}{1}
\tableofcontents

\section{Introduction}
Let $G$ be a complex reductive group and $N$ be a complex representation of $G$. To the pair $(G, N\oplus N^*)$, we can define the associated Higgs branch by the Hamiltonian reduction procedure. In a seminal work of Nakajima and Braverman--Finkelberg--Nakajima \cite{Na16,BFN18}, they gave a mathematically rigorous definition of the corresponding Coulomb branch using the affine Grassmannian. One of the advantages of their definition is that the Coulomb branch algebra naturally comes with a quantization by considering the loop rotation. Since then, the Coulomb branch has found a lot of applications in other fields of mathematics and physics; for example, see the following survey papers \cite{BF19,FinICM,Kam22}. 

Among the many examples, the quiver gauge theory is of particular interest. Let $Q=(Q_0,Q_1)$ be a quiver, where $Q_0$ is the set of vertices while $Q_1$ is the set of edges. For each edge $h\in Q_1$, let $s(h)$ (resp. $t(h)$) denote the source (resp. target) of the edge. Let $V$ and $W$ be two $Q_0$-graded vector spaces, and consider
\[
    G_V:=\prod_{i\in Q_0}\GL(V_i),\quad E_V:=\bigoplus_{h\in Q_1}\Hom(V_{s(h)}, V_{t(h)}),\textit{ and } L_{W,V}:=\bigoplus_{i\in Q_0}\Hom(W_i,V_i).
\]
Then the data 
\[
    (G,N):=(G_V,E_V\oplus L_{W,V})
\]
defines the so-called quiver gauge theory, whose Higgs branch is the Nakajima quiver variety \cite{Nak98}, while the Coulomb branch is the generalized affine Grassmannian slices \cite{BFN19}. Quiver varieties played an important role in the geometric study of quantum affine algebras and Yangians; see \cite{NakICM,Nak01,Var00}. On the other hand, the quantized Coulomb branch algebra of the quiver gauge theory can be used to give geometric realizations of the shifted quantum affine algebras and the shifted Yangians; see \cite{KWWY14,BFN19,We19,FT19}.

In this paper, we study quivers with involution $(Q,\tau)$, where $Q$ is simply-laced quiver without self-loops and $\tau$ is an involution of $Q$ satisfying
\begin{itemize}
    \item 
    $s(\tau(h))=\tau(t(h))$ and $ t(\tau(h))=\tau(s(h))$; 
    \item $\tau(s(h))=t(h)$ if and only if $\tau(h)=h$;
    \item $\tau(i)\neq i$ for any $i\in Q_0$.
\end{itemize}
This is considered by Enomoto--Kashiwara and Varagnolo--Vasserot to generalize a result of Ariki about affine Hecke algebra of type A; see \cite{Ari,EK07,EK08,EK082,VV11}. On the other hand, Satake diagrams, which are used to classify the symmetric pairs, are examples of quivers with involutions. Here is an example of type AIII$_6^{(\tau)}$:
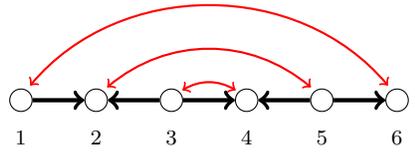
\[
    \begin{tikzpicture}[scale=1, 
                        arrow/.style={red, <->, thick}]
        \foreach \i in {0,...,5}
            \node[circle, draw, inner sep=3pt] (a\i) at (\i,0) {};
            \foreach \i in {0,...,5}
            \node (b\i) at (\i,.05) {};
            \foreach \i in {1,...,6}
            \node  at (\i-1,-.5) {$\scriptstyle \i$};
        \draw[arrow] (b0) to[bend left=50] (b5);
        \draw[arrow] (b1) to[bend left=40] (b4);
        \draw[arrow] (b2) to[bend left=30] (b3);
        \draw[->,ultra thick] (a0) -- (a1); 
        \draw[->,ultra thick] (a2) -- (a1); 
        \draw[->,ultra thick] (a2) -- (a3); 
        \draw[->,ultra thick] (a4) -- (a3); 
        \draw[->,ultra thick] (a4) -- (a5); 
    \end{tikzpicture},
\]
where the red arrow denotes the involution. Quantizations of the symmetric pairs give the iquantum groups/twisted Yangians, which are coideal subalgebras of the usual quantum groups/Yangians; see \cite{Wan23} for a survey. Utilizing the Lu--Wang--Zhang's Drinfeld new presentation for the affine iquantum group of type AIII$_{2n-1}^{(\tau)}$ \cite{LWZ24}, Su and Wang gave a geometric realization of it via the equivariant K-theory of the Steinberg variety for the cotangent bundle of partial flag varieties in type $C$ \cite{SuW24}, which is an example of the $\sigma$-quiver variety introduced in \cite{Li19}. On the other hand, Nakajima \cite{nakajima2025instantons} was able to compute the K-matrices for some special $\sigma$-quiver varieties \cite{Li19} and got representations of the twisted Yangians on the equivariant cohomology of the $\sigma$-quiver varieties. His approach completely avoids the Drinfeld new presentation and instead relies on the Maulik--Okounkov stable envelopes \cite{MO19}. These works can be regarded as geometric realizations of the quantum symmetric pairs from the Higgs branch side.

The study of the corresponding Coulomb branch associated to $(Q,\tau)$ was initiated by Lu--Wang--Weekes \cite{LWW,LWW2}, see also \cite{bartlett2025gklo,nakajima2025instantons,SSX25} for related works. This paper is a continuation of \cite{SSX25}. Assume $\bigoplus_{i\in Q_0}V_i$ is equipped with a nondegenerate symmetric bilinear form $\langle-,-\rangle$ such that the orthogonal complement $V_i^\perp=\bigoplus_{j\neq \tau i} V_j$. Hence, we can identify $V_i=V_{\tau i}^*$. Let
\begin{align*}
    G_V^{\tau} &:=\left\{g\in G_V\mid g_i^t = g_{\tau i}^{-1}, \forall i\in Q_0\right\}\subseteq G_V,\\
    E_V^{\tau,-}   &:=\left\{f\in E_V\mid f_{\tau h}=-f_h^t, \forall h\in Q_1 \right\}\subseteq E_V.
\end{align*}
In particular, if an edge $h\in Q_1$ is fixed by $\tau$, then 
\[
    f_h\in \wedge^2 (V_{t(h)}),
\]
the second wedge power of $V_{t(h)}$. Shen--Su--Xiong considered the quantized Coulomb branch associated to the following data 
\[
    G=G_V^{\tau},\qquad \bfN = E_V^{\tau,-} \oplus L_{W,V},\qquad F=G_W,
\]
where $G_W:=\prod_{i\in Q_0}\GL(W_i)$ is the flavor symmetry group, and established an algebra homomorphism from the shifted twisted Yangians $\mathbf{Y}^{\tau}_{\mu}(\mathfrak{g})$ to the quantized Coulomb branch algebra. 

The above case was also proposed by Nakajima \cite[Section 3(ix)(C)]{nakajima2025instantons}. Moreover, he also proposed the second symmetric power case. I.e., let 
\[
    E_V^{\tau}:=\left\{f\in E_V\mid f_{\tau h}=f_h^t, \forall h\in Q_1 \right\}\subseteq E_V.
\]
Hence, for an edge $h\in Q_1$ fixed under the involution,
\[
    f_h\in \mathsf{S}^2 (V_{t(h)}),
\]
the second symmetric power of $V_{t(h)}$.
 
Recently, Lu and Zhang found the Drinfeld new presentation for all the quasi-split twisted Yangians \cite{LZ24}. Based on this, we can define the shifted twisted Yangians $\mathbf{Y}^{\tau}_{\mu}(\mathfrak{g})$ for any $\tau$-invariant coweight $\mu$. The following is the main result of this paper.
\begin{theorem*}[\cref{thm:main}]
    There is an algebra homomorphism from the shifted twisted Yangians $\mathbf{Y}^{\tau}_{\mu}(\mathfrak{g})\otimes H_{G_W}^*(\pt)$ to the quantized Coulomb branch algebra associated to
    \[
        (G=G_V^{\tau},\qquad \bfN = E_V^\tau \oplus L_{W,V},\qquad F=G_W).
    \]
\end{theorem*}
As in \cite{SSX25}, we first establish a Gerasimov--Kharchev--Lebedev--Oblezin (GKLO)-type representation \cite{GKLO05} for $\mathbf{Y}^{\tau}_{\mu}(\mathfrak{g})$ via some difference operators, see \cref{thm:gklo}. On the other hand, Braverman--Finkelberg--Nakajima showed that the quantized Cou\-lomb branch algebra can be embedded into some difference algebra, and the image for the dressed minuscule operators can be computed explicitly. With these, we show that the GKLO-type representation factors through the quantized Coulomb branch algebra, thus proving the main result.

\subsection*{Acknowledgment} This work builds directly upon the results of \cite{SSX25}. The author is grateful to Shen, Su, and Xiong for generously explaining their work. He also sincerely thanks Professor Nakajima for posing the problem of calculating the quantized Coulomb branch algebra for the second symmetric power case. The author is supported by the National Key R\&D Program of China (No. 2024YFA1014700) from his advisor C. Su.

\section{Coulomb branches of cotangent type}

\subsection{BFN Coulomb Branch}
We give a brief review of the Coulomb branch in this section, following the notations in \cite{BFN18}. Let $G$ be a complex reductive group with a Borel subgroup $B$ and a maximal torus $T$ with Weyl group $W$. Let $X_*(T)$ be the cocharacter lattice of $T$ with dominant ones denoted by $X_*(T)^+$. For any $\lambda,\mu\in X_*(T)^+$, $\lambda\leq \mu$ iff $\mu-\lambda$ is a nonnegative linear combination of the positive coroots. Let $\calO=\mathbb{C}\llbracket z \rrbracket $ and $\calK=\mathbb{C}(\!(z)\!)$. We write $G_\calK:=G(\calK)$, $G_\calO:=G(\calO)$ and $T_\calK:=T(\calK)$, respectively. Any cocharacter $\lambda:\mathbb{G}_m\rightarrow T$ give a homomorphism $\mathcal{K}^*\rightarrow T_\mathcal{K}$, and we let $z^\lambda$ denote the image of $z\in \mathcal{K}^*$.

Recall the affine Grassmannian is defined to be $\Gr_G=G_\calK/G_\calO$. By the Cartan decomposition,
\[
    \Gr_G=\bigsqcup_{\lambda\in X_*(T)^+}\Gr_G^\lambda,
\]
where $\Gr_G^\lambda:=G_\calO z^\lambda G_\calO/G_\calO$. Its closure is $\bGr_G^\lambda=\bigsqcup_{\mu\in X_*(T)^+,\mu\leq \lambda}\Gr_G^\mu$.

Let $\bfN$ be a complex representation of $G$. Let $\bfN_\calK:=\bfN(\calK)$ and $\bfN_\calO:=\bfN(\calO)$, respectively. Recall the Braverman--Finkelberg--Nakajima (BFN) space is
\[
    \calR_{G,\bfN} = \{(gG_\calO,x): x\in g\bfN_\calO\cap \bfN_\calO\}\subset \Gr_G\times \bfN_\calO.
\]
Following \cite[Section 2(i)]{BFN18}, we consider the $\mathbb{C}^\times$ action on $\calR_{G,\bfN}$, which rotates $z\in \calO$ by weight $1$ and scales $\bfN$ by weight $\frac{1}{2}$ simultaneously. Let $\hbar$ be the equivariant parameter of this $\mathbb{C}^*$, and denote this $\mathbb{C}^*$ by $\mathbb{C}^\times_\hbar$. The \emph{(quantized) Coulomb branch algebra} is defined to be the Borel--Moore homology
\[
    \mathcal{A}_\hbar(G,\bfN) = H^{G_{\calO}\rtimes\mathbb{C}_\hbar^\times}_*(\calR_{G,\bfN}).
\]
By \cite{BFN18}, $\mathcal{A}_\hbar(G,\bfN)$ admits a convolution product. 

Let $\mathfrak{t}$ be the Lie algebra of $T$, and $\mathbb{A}^1$ be the Lie algebra of $\mathbb{C}^*_\hbar$. We define the \emph{difference algebra} $\Diff_\hbar(T)$ of $T$ to be 
\[
    \Diff_\hbar(T)=\mathbb{C}(\mathfrak{t}\times \mathbb{A}^1)\rtimes X_*(T),
\]
where $\mathbb{C}(\mathfrak{t}\times \mathbb{A}^1)$ is the field of rational functions over $\mathfrak{t}\times \mathbb{A}^1$. 
I.e.,
\[
    (f(t,\hbar)d_\lambda)\cdot (g(t,\hbar)d_\mu)=f(t,\hbar)g(t+\lambda\hbar,\hbar)d_{\lambda+\mu}.
\]
By \cite{BFN18}, there is an embedding of algebras
\begin{equation}\label{eq:embeddiff}
\varphi:\mathcal{A}_\hbar(G,\bfN)\hookrightarrow \Diff_\hbar(T). 
\end{equation}
Therefore, we can view elements in the quantized Coulomb branch algebra $\mathcal{A}_\hbar(G,\bfN)$ as some difference operators. For the later applications, we need to review some explicit formula for some special elements in $\mathcal{A}_\hbar(G,\bfN)$.

Recall there is a natural projection 
\[
    \pi:\calR_{G,\bfN}\longrightarrow \Gr_G.
\]
Let us denote 
$ \calR_{\lambda}=\pi^{-1}(\Gr_G^\lambda) $ and $ \calR_{\leq \lambda} = \pi^{-1}(\bGr_G^\lambda) $. The quantized Coulomb branch algebra $\mathcal{A}_\hbar(G,\bfN)$ is filtered by dominant coweights, and the associated graded algebra is
\[
    \gr \mathcal{A}_\hbar(G,\bfN)\simeq\bigoplus_{\lambda\in X_*(T)^+} H_*^{G_\calO\rtimes\mathbb{C}^\times_\hbar}(\calR_\lambda)\simeq\bigoplus_{\lambda\in X_*(T)^+} \mathbb{C}[\mathfrak{t}\times\mathbb{A}^1]^{W_\lambda}[\calR_\lambda],
\]
where $[\calR_\lambda]$ is the fundamental class of $\calR_{\lambda}$. 

For a minuscule cocharacter $ \lambda $, $\Gr_\lambda=\overline{\Gr_\lambda}=G/P_\lambda$, where $P_\lambda$ is a parabolic subgroup whose associated Weyl group is $W_\lambda$, the stabilizer of $\lambda$ inside $W$. Hence, $\calR_{\lambda}$ is also closed. Moreover,
\[
    H_*^{G_\calO\rtimes\mathbb{C}^\times_\hbar}(\calR_\lambda) \simeq H_*^{G_\calO\rtimes\mathbb{C}^\times_\hbar}(\Gr_\lambda) \simeq \mathbb{C}[\mathfrak{t}\times\mathbb{A}^1]^{W_\lambda}.
\]
For $f\in \mathbb{C}[\mathfrak{t}\times\mathbb{A}^1]^{W_\lambda}\simeq H_*^{G_\calO\rtimes\mathbb{C}^\times_\hbar}(\calR_\lambda)$, the element $ f[\calR_\lambda]\in \gr\mathcal{A}_\hbar(G,\bfN) $ lifts to an element in $\mathcal{A}_\hbar(G,\bfN)$, and via the embedding \eqref{eq:embeddiff}, it is sent to
\begin{equation}\label{eq:monopoleop}
    \varphi(f[\calR_\lambda])=\sum_{w\in W^\lambda}w\left(f\cdot \frac{\Eu\left(z^\lambda\bfN_\calO/(z^\lambda\bfN_\calO\cap\bfN_\calO)\right)}{\Eu(T_{\lambda}\Gr_\lambda)}\right)d_{w\lambda},
\end{equation}
where $W^\lambda$ is the set of minimal length representatives of cosets in $W/W_\lambda$,  $\Eu(-)$ is the $T\times\mathbb{C}^*_\hbar$-equivariant Euler class, i.e. the product of weights, and $T_{\lambda}\Gr_\lambda$ is the tangent space of $\Gr_\lambda$ at the torus fixed point $z^\lambda G_\calO/G_\calO$. The operators $f[\calR_\lambda]$ are called \textit{dressed monopole operators}. We refer to \cite[Section 2.1]{SSX25} for more details about the computation of these Euler classes.

Let $F$ be a reductive group and $\tilde{G}=G\times F$. We assume further that the $G$-representation $\bfN$ can be extended to a $\tilde{G}$-representation. Then we can slightly extend the above definition by 
\[
    \mathcal{A}_\hbar := H^{(G_{\calO}\times F_{\calO})\rtimes\mathbb{C}_\hbar^\times}_*(\calR_{G,\bfN}).
\]
The dependence of $F$ will be clear from the context. 
Then the embedding \eqref{eq:embeddiff} becomes
\[
    \mathcal{A}_\hbar\hookrightarrow H_F^*(\pt)\otimes\Diff_\hbar(T).
\]
The results in this section still hold.

\subsection{Quiver with involution}\label{sec:quiverinv}
Let $Q=(Q_0,Q_1,s,t)$ be a quiver, where $Q_0$ is the set of vertices while $Q_1$ is the set of arrows. For any $h\in Q_1$, $s(h)$ (resp. $t(h)$) denotes the source (resp. target) of the edge $h$. For two $Q_0$-graded vector spaces $W,V$ with dimension vectors $\mathbf{w}=(w_i)_{i\in Q_0},\mathbf{v}=(v_i)_{i\in Q_0}$, let us denote 
\begin{align*}
    G_V &= \prod_{i\in Q_0}GL(V_i), & G_W &= \prod_{i\in Q_0}GL(W_i),\\
    E_V &= \bigoplus_{h\in Q_1}\Hom(V_{s(h)},V_{t(h)}), & L_{W,V} &= \bigoplus_{i\in Q_0}\Hom(W_i,V_i).
\end{align*}

We now assume that there is an involution $\tau$ on $Q=(Q_0,Q_1,s,t)$ satisfying
\begin{itemize}
    \item 
    $s(\tau(h))=\tau(t(h))$ and $ t(\tau(h))=\tau(s(h))$; 
    \item $\tau(s(h))=t(h)$ if and only if $\tau(h)=h$. 
\end{itemize}

The pair $(Q,\tau)$ is called a \emph{quiver with involution} in the literature; see \cite{EK07} and \cite{VV11}. Assume $\bigoplus_{i\in Q_0}V_i$ is equipped with a nondegenerate symmetric bilinear form $\langle-,-\rangle$ such that the orthogonal complement $V_i^\perp=\bigoplus_{j\neq \tau i} V_j$. In particular, $\langle-,-\rangle$ restricts to a perfect pairing between $V_i$ and $V_{\tau i}$. Hence we can identify $V_i=V_{\tau i}^*$ and $\mathbf{v}$ is $\tau$-invariant. Let us denote 
\begin{align*}
    G_V^{\tau} &=\left\{g\in G_V\mid g_i^t = g_{\tau i}^{-1}, \forall i\in Q_0\right\}\subseteq G_V,\\
    E_V^\tau   &=\left\{f\in E_V\mid f_{\tau h}=f_h^t, \forall h\in Q_1 \right\}\subseteq E_V.
\end{align*}
Hence, $G_V^\tau$ acts on $E_V^\tau$. Similar to \cite{SSX25}, we further assume the quiver $Q$ is simply-laced without self-loop, and $ \tau $ has no fixed points in $ Q_0 $. With these assumptions, we will study the Coulomb branch algebra for
\[
    G=G_V^{\tau},\qquad \bfN = E_V^\tau \oplus L_{W,V},\qquad F=G_W.
\]

In order to make the computation more precise, we can pick a decomposition $Q_0=Q_0^{+}\sqcup Q_0^-$ such that $i\in Q_0^{+}$ if and only if $\tau i\in Q_0^-$. Then the composition 
\[
    G_V^{\tau}\hookrightarrow G_V\stackrel{\text{pr}}\longrightarrow \prod_{i\in Q_0^+}GL(V_i)
\]
is an isomorphism. Let $Q_1^\tau$ be the subset of edges in $Q_1$, which are fixed by the involution $\tau$. For a decomposition $Q_1\setminus Q_1^\tau=Q_1^+ \sqcup Q_1^-$ such that $h\in Q_1^{+}$ if and only if $\tau(h)\in Q_1^-$, the composition
\[
    E_V^\tau \hookrightarrow E_V
    \stackrel{\text{pr}}\longrightarrow 
    \bigoplus_{h\in Q_1^+}
    \Hom(V_{s(h)},V_{t(h)})
    \oplus 
    \bigoplus_{h\in Q_1^\tau}
    \mathsf{S}^2 V_{t(h)}
\]
is an isomorphism, where $\mathsf{S}^2 V_{t(h)}$ can be viewed as a quotient of $\Hom(V_{s(h)},V_{t(h)})$ if $h\in Q_1^\tau$. 

\begin{remark}
    The case 
    \begin{align*}
        E_V^{\tau,-}:=&\left\{f\in E_V\mid f_{\tau h} = -f_h^t, \forall h\in Q_1 \right\}\\
        \simeq& \bigoplus_{h\in Q_1^+}\Hom(V_{s(h)},V_{t(h)})\oplus \bigoplus_{h\in Q_1^\tau} \wedge^2 V_{t(h)}
    \end{align*} is studied in \cite{SSX25}. Both the second wedge power and the second symmetric power cases are considered by Nakajima, see \cite[Section 3(ix)(C)]{nakajima2025instantons}.
\end{remark}

\begin{example}[Type AIII]
    Consider a quiver with involution obtained from a Satake diagram of type AIII (see \cite{Ara62}), i.e. 
    \[
        Q_0=\{1,\ldots,2n\},\qquad\tau i = 2n+1-i
    \]
    and $Q_1$ is chosen such that 
    \[
    \texttt{\#}\big\{h\in Q_1:\{s(h),t(h)\}=\{i,j\}\big\}=\delta_{|i-j|,1}.
    \]
    We can choose $Q_0^+=\{1,\ldots,n\}$ and $Q_1^+=\{h\in Q_1: \max\{s(h),t(h)\}\leq n\}$. For example, when $n=3$ we have
    
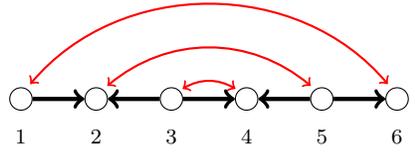
\[
        \begin{tikzpicture}[scale=1, 
                            arrow/.style={red, <->, thick}]
            \foreach \i in {0,...,5}
                \node[circle, draw, inner sep=3pt] (a\i) at (\i,0) {};
                \foreach \i in {0,...,5}
                \node (b\i) at (\i,.05) {};
                \foreach \i in {1,...,6}
                \node  at (\i-1,-.5) {$\scriptstyle \i$};
            \draw[arrow] (b0) to[bend left=50] (b5);
            \draw[arrow] (b1) to[bend left=40] (b4);
            \draw[arrow] (b2) to[bend left=30] (b3);
            \draw[->,ultra thick] (a0) -- (a1); 
            \draw[->,ultra thick] (a2) -- (a1); 
            \draw[->,ultra thick] (a2) -- (a3); 
            \draw[->,ultra thick] (a4) -- (a3); 
            \draw[->,ultra thick] (a4) -- (a5); 
        \end{tikzpicture},
    \]
    where the red arrow denotes the involution. Then we have 
    \begin{align*}
        G_V^{\tau} & \cong GL(v_1)\times GL(v_2)\times GL(v_3)\\
        E_V^\tau & \cong \Hom(\mathbb{C}^{v_1},\mathbb{C}^{v_2})\oplus 
        \Hom(\mathbb{C}^{v_3},\mathbb{C}^{v_2})\oplus 
        \mathsf{S}^2(\mathbb{C}^{v_3}).
    \end{align*}
\end{example}

Let us pick a basis $\{e_{i,1},\ldots,e_{i,v_i}\}$ of each $V_i$ such that $\langle e_{i_1,j_1},e_{i_2,j_2}\rangle = \delta_{i_1,\tau i_2 }\delta_{j_1,j_2}$. This gives a choice of a maximal torus $T_V^\tau$ of $G_V^{\tau}$. Then we can identify 
\[
    H_{T_V^\tau}^*(\pt) = \bigotimes_{i\in Q_0^+}\mathbb{Q}[x_{i,1},\ldots,x_{i,v_i}],
\]
and
\[
    X_*(T^\tau_V) = \bigoplus_{i\in Q_0^+}\mathbb{Z}\epsilon_{i,1} \oplus \cdots \oplus \mathbb{Z}\epsilon_{i,v_i}.
\]
Here $\epsilon_{i,j}$ is the $j$-th coordinate on the subtorus $T_{V_i}$ and $x_{i,j}$ is the $j$-th character of $T_{V_i}$. A cocharacter $\lambda=\sum\lambda_{i,j}\epsilon_{i,j}\in X_*(T^\tau_V)$ is dominant if $\lambda_{i,1}\geq \cdots \geq \lambda_{i,v_i}$ for each $i\in Q_0^+$. For $i\in Q_0^+$ and $1\leq j\leq v_i$, we denote the difference operator of $\epsilon_{i,j}$ by $d_{i,j}\in \Diff_\hbar(T_V^\tau)$. Let us introduce $x_{i,j}=-x_{\tau i,j}$ and $d_{i,j}=d_{\tau i,j}^{-1}$ for $i\in Q_0^-$. Hence, we have the following relation in $\Diff_\hbar(T_V^\tau)$:
\[
    d_{i_1,j_1}x_{i_2,j_2}=(x_{i_2,j_2}+(\delta_{i_1,i_2}-\delta_{i_1,\tau i_2 })\delta_{j_1,j_2}\hbar)d_{i_1,j_1}
\]
for any $i_1,i_2\in Q_0$, $1\leq j_1\leq v_{i_1}$ and $1\leq j_2\leq v_{i_2}$. Let us pick a basis $\{f_{i,1},\ldots,f_{i,w_i}\}$ of each $W_i$. This gives a choice of a maximal torus $T_W$ of $G_W$, and we can identify 
\[
    H_{G_W}^*(\pt) = \bigotimes_{i\in Q_0}\mathbb{Q}[w_{i,1},\ldots,w_{i,w_i}]^{S_{w_i}},
\]
where $S_{w_i}$ is the symmetric group.

\subsection{Monopole operators}\label{sec:mono}
In this section, we give explicit formulae for some monopole operators of the quantized Coulomb branch algebra
\[
    \mathcal{A}_\hbar := H^{(G_{\calO}\times F_{\calO})\rtimes\mathbb{C}_\hbar^\times}_*(\calR_{G,\bfN})
\]
associated to $ G=G_V^{\tau},\bfN = E_V^\tau \oplus L_{W,V},F=G_W $. Recall that it can be embeded into the difference algebra $H_{G_W}^*(\pt)\otimes\Diff_\hbar(T_V^\tau)$.

For $i\in Q_0$, define the following polynomials
\[
    V_i(z) = \prod_{k=1}^{v_i}(z-x_{i,k}), \quad W_i(z) = \prod_{k=1}^{w_i}(z-w_{i,k}).
\]
Moreover, for $1\leq r\leq v_i$, let
\[
    V_{i,r}(z)=\frac{V_i(z)}{z-x_{i,r}}.
\]

For $i\in Q_0^+$, the dominant coweight $\epsilon_{i,1}\in X_*(T_V^\tau)$ is a minuscule coweight for $G$, and the spherical Schubert cell $\Gr_{\epsilon_{i,1}}$ is closed and isomorphic to the projective space $\mathbb{P}^{v_i-1}$. More precisely, it is identified with the moduli space of $\calO$-modules $L$ such that 
\[
    z\calO\otimes V_i\subset L\subset \calO\times V_i,\quad \dim_\mathbb{C}\calO\otimes V_i/L =1.
\]
There is a tautological line bundle on $\Gr_{\epsilon_{i,1}}$ whose fiber at $L$ is $\calO\otimes V_i/L$. Thus, the torus weight of this line bunlde at the fixed point $z^{\epsilon_{i,r}}$ is $x_{i,r}$. Let $\mathcal{Q}_i$ denote the pullback of this line bundle to $\calR_{\epsilon_{i,1}}$.

By the same argument as in \cite[Proposition 2.4]{SSX25}, we get the following explicit formulae.
\begin{prop}\label{prop:fmono}
    Let $f\in \mathbb{Q}[x]$ be a polynomial in one variable. We have 
    \begin{align*}
        &f(c_1(\mathcal{Q}_i)) \cap [\mathcal{R}_{\epsilon_{i,1}}]
        =\sum_{r=1}^{v_i}f(x_{i,r})
        \prod_{\substack{h\in Q_1\setminus Q_1^\tau \\ s(h)=i}} (-1)^{v_{t(h)}}V_{t(h)}(x_{i,r}+\tfrac{\hbar}{2})\\
        &\cdot \prod_{\substack{h\in Q_1^\tau \\ s(h)=i}}
        (-1)^{v_i-1}(2x_{i,r}+\tfrac{3\hbar}{2})
        V_{t(h)}(x_{i,r}+\tfrac{\hbar}{2}). \frac{W_{\tau i}(x_{\tau i,r}-\tfrac{\hbar}{2})}{V_{i,r}(x_{i,r})}d_{i,r}\in H_{G_W}^*(\pt)\otimes\Diff_\hbar(T_V^\tau).
    \end{align*}
\end{prop}

Now let us consider the case of $\lambda=-\epsilon_{i,v_i}\in X_*(T_V^\tau)$, where $i\in Q_0^+$. This is also a minuscule coweight, and the spherical Schubert cell $\Gr_{-\epsilon_{i,v_i}}$ is also isomorphic to the projective space $\mathbb{P}^{v_i-1}$. More precisely, it is identified with the moduli space of $\calO$-modules $L$ such that 
\[
    \calO\otimes V_i\subset L\subset z^{-1}\calO\times V_i,\quad \dim_\mathbb{C}L/(\calO\otimes V_i) =1.
\]
There is a tautological line bundle on $\Gr_{-\epsilon_{i,v_i}}$ whose fiber at $L$ is $L/(\calO\otimes V_i)$. Therefore, the torus weight of the fiber at the torus fixed point $z^{-\epsilon_{i,r}}$ is $x_{i,r}-\hbar$. Let $\mathcal{S}_i$ denote the pullback of this line bundle to $\calR_{-\epsilon_{i,v_i}}$. Similar as above, we get
\begin{prop}\label{prop:fmono2}
    Let $f\in \mathbb{Q}[x]$ be a polynomial in one variable. We have 
    \begin{align*}
        &f(c_1(\mathcal{S}_i)) \cap [\mathcal{R}_{-\epsilon_{i,v_i}}]
        =\sum_{r=1}^{v_i}f(x_{i,r}-\hbar)
        \prod_{\substack{h\in Q_1\setminus Q_1^\tau \\ s(h)=\tau i}}(-1)^{v_{t(h)}}V_{t(h)}(x_{\tau i,r}+\tfrac{\hbar}{2})\\
        &\cdot\prod_{\substack{h\in Q_1^\tau \\ t(h)=i}}
        (-1)^{v_i}(2x_{i,r}-\tfrac{3\hbar}{2})
        V_{i}(x_{\tau i,r}+\tfrac{\hbar}{2})\frac{W_{i}(x_{i,r}-\tfrac{\hbar}{2})}{(-1)^{v_i-1}V_{i,r}(x_{i,r})}d_{i,r}^{-1}\in H_{G_W}^*(\pt)\otimes\Diff_\hbar(T_V^\tau).
    \end{align*}
\end{prop}

\section{Shifted twisted Yangians}\label{sec:Yangian}

Let $C = (c_{ij})_{i,j \in I}$ be a symmetric generalized Cartan matrix, and let $\mathfrak{g}$ be the associated Kac--Moody Lie algebra with Cartan subalgebra $\mathfrak{h}$. We further make the following assumptions: 
\begin{itemize}
    \item there exists an involution $\tau: I \to I$ such that $c_{ij} = c_{\tau i, \tau j}$ for all $i,j \in I$;
    \item $c_{ij} = c_{ji} \in \{0,-1\}$, for any $i \neq j$;
    \item $i \neq \tau i$ for all $i \in I$.
\end{itemize}

Let $\{\alpha_i\in \mathfrak{h}^* \mid i \in I\}$ (resp. $\{\alpha_i^\vee\in \mathfrak{h} \mid i \in I\}$) be the set of simple roots (coroots) associated with $\mathfrak{g}$. Let $\Lambda_i\in \mathfrak{h}$ be the fundamental coweights satisfying
\[
    \langle  \Lambda_i,\alpha_j  \rangle = \delta_{ij}, \qquad i,j \in I,
\]
where $\langle-,-\rangle$ is the natural pairing between $\mathfrak{h}$ and $\mathfrak{h}^*$. Let $P^\vee\subset \mathfrak{h}$ be the lattice generated by $\alpha_i^\vee$ and $\Lambda_i$. The involution $\tau$ on $I$ naturally extends to $P^\vee$ by setting $\tau(\Lambda_i) = \Lambda_{\tau i}$ and $\tau(\alpha_i^\vee)=\alpha_{\tau i}^\vee$.

The definition of the shifted twisted Yangian depends on a choice of a $\tau$-invariant coweight $\mu \in P^\vee$, i.e.\ $\mu = \tau(\mu)$. In order to write down explicit generators, we further fix a family of parameters $\zeta = (\zeta_i)_{i \in I} \in \bigl(\mathbb{C}(\hbar)^\times\bigr)^I$ satisfying 
\begin{equation}\label{eq:zeta}
  \zeta_i = (-1)^{\langle \alpha_i, \mu \rangle}\zeta_{\tau(i)} \in \mathbb{C}(\hbar)^\times.
\end{equation}

Due to the Drinfeld new presentation for twisted Yangian found by Lu and Zhang \cite{LZ24}, we make the following definition, see also \cite{SSX25}. 
\begin{definition}\label{shiftedTY}
    The shifted twisted Yangian $\mathbf{Y}^{\tau}_{\mu}(\mathfrak{g})$ is the $\mathbb{C}(\hbar)$-algebra generated by 
    $h_{i,r}$ and $b_{i,s}$ for $i\in I, r\geq -\langle \alpha_i,\mu\rangle-1$ and $s\geq 0$, subject to the following relations:
    \begin{align}
    \label{eq:hcomm}
        & {\left[h_{i, r}, h_{j, s}\right]=0, \qquad 
        h_{i,s} = (-1)^{s+1} h_{\tau i,s}}, 
            \\
    \label{eq:hbcomm}
        & {\left[h_{i, r+2}, b_{j, s}\right]-\left[h_{i, r}, b_{j, s+2}\right]} 
        =\frac{c_{i j}-c_{\tau i, j}}{2} \hbar\left\{h_{i, r+1}, b_{j, s}\right\}
        \\\notag  & \qquad 
        +\frac{c_{i j}+c_{\tau i, j}}{2} \hbar\left\{h_{i, r}, b_{j, s+1}\right\}
        +\frac{c_{i j} c_{\tau i, j}}{4} \hbar^2\left[h_{i, r}, b_{j, s}\right], 
            \\
    \label{eq:bcomm}
        & {\left[b_{i, r+1}, b_{j, s}\right]-\left[b_{i, r}, b_{j, s+1}\right]=\frac{c_{i j}}{2}\hbar \left\{b_{i, r}, b_{j, s}\right\}-2 \delta_{\tau i, j}(-1)^r h_{j, r+s+1}},
    \end{align}
    and the Serre-type relations  
    \begin{itemize}
        \item when $c_{i j}=0$, we have
        \begin{equation}\label{eq:commSerre}
            \left[b_{i, r}, b_{j, s}\right]=\delta_{\tau i, j}(-1)^r h_{j, r+s};
        \end{equation}
        \item when $c_{i j}=-1, j \neq \tau i \neq i$, we have
        \begin{equation}\label{eq:usualSerre}
            \operatorname{Sym}_{k_1, k_2}\left[b_{i, k_1},\left[b_{i, k_2}, b_{j, r}\right]\right]=0;
        \end{equation}
        \item when $c_{i, \tau i}=-1$, we have
        \begin{equation}\label{eq:iSerre}
            \operatorname{Sym}_{k_1, k_2}\left[b_{i, k_1},\left[b_{i, k_2}, b_{\tau i, r}\right]\right]=\frac{4}{3} \operatorname{Sym}_{k_1, k_2}(-1)^{k_1} \sum_{p=0}^{\infty} 3^{-p}\left[b_{i, k_2+p}, h_{\tau i, k_1+r-p}\right].
        \end{equation}
    \end{itemize}
    Here we have adopted the convention that 
    \begin{equation}\label{eq:vanishing}
        \mbox{$h_{i,r}=0$ if $r<-\langle \alpha_i,\mu\rangle-1$ and $h_{i,r}=\zeta_i$ if $r=-\langle \alpha_i,\mu\rangle-1$}
    \end{equation}
\end{definition}
Note that the right-hand side of \cref{eq:iSerre} is a finite sum by the convention \cref{eq:vanishing}. We define the following generating functions
\begin{align}
    \label{generatingfunction}
    h_i(z) & = \hbar\sum_{r\in \mathbb{Z}}h_{i,r}z^{-r-1},&
    h_i^\circ(z) & = \hbar\sum_{r\geq 0}h_{i,r}z^{-r-1},&
    b_i(z) & = \hbar\sum_{s\geq 0}b_{i,s}z^{-s-1}.
\end{align}
Following \cite[\S 3.3]{LZ24}, we can rewrite the relations \eqref{eq:hcomm}, \eqref{eq:hbcomm} and \eqref{eq:bcomm} as 
\begin{align}
    \label{eq:hcomm(z)}&
    {\left[h_i(u), h_j(v)\right]=0, \qquad h_{\tau i}(u)=h_i(-u),}\\
    \label{eq:hbcomm(z)}
    &
    \left(u^2-v^2\right)\left[h_i(u), b_j(v)\right]=\frac{c_{i j}-c_{\tau i, j}}{2} \hbar u\left\{h_i(u), b_j(v)\right\}\\
    \notag& \qquad +\frac{c_{i j}+c_{\tau i, j}}{2} \hbar v\left\{h_i(u), b_j(v)\right\} +\frac{c_{i j} c_{\tau i, j}}{4} \hbar^2\left[h_i(u), b_j(v)\right]
    \\\notag&\qquad 
    -\hbar\left[h_i(u), b_{j, 1}\right] 
    -\hbar v\left[h_i(u), b_{j, 0}\right]-\frac{c_{i j}+c_{\tau i, j}}{2} \hbar^2\left\{h_i(u), b_{j, 0}\right\},\\
    \label{eq:bcomm(z)}&
    (u-v)\left[b_i(u), b_j(v)\right]=\frac{c_{i j}}{2} \hbar\left\{b_i(u), b_j(v)\right\}+\hbar\left(\left[b_{i, 0}, b_j(v)\right]-\left[b_i(u), b_{j, 0}\right]\right) \\
    \notag&\qquad -\delta_{\tau i, j} \hbar\left(\frac{2 u}{u+v} h_i^{\circ}(u)+\frac{2 v}{u+v} h_j^{\circ}(v)\right).
\end{align}
We can also write two of Serre relations \eqref{eq:commSerre} and \eqref{eq:usualSerre} as 
\begin{align}
    \label{eq:commSerre(z)}&
    (u+v)\left[b_i(u), b_j(v)\right]=\delta_{\tau i, j}\hbar\left(h_j^\circ(v)-h_i^\circ(u)\right)&(c_{ij}=0),\\
    \label{eq:usualSerre(z)}&
    \operatorname{Sym}_{u_1, u_2}\left[b_i(u_1),\left[b_i(u_2), b_j(v)\right]\right]=0, &
    (c_{i j}=-1, j \neq \tau i \neq i).
\end{align}

\section{Main Results}
In this section, we first produce a GKLO type representation of the shifted twisted Yangian via difference operators. Then we prove our main result showing that for a quiver with involution, there is an algebra homomorphism from the shifted twisted Yangian to the quantized Coulomb branch algebra associated to the corresponding involution-fixed part of the quiver gauge theory.

Let $(Q,\tau)$ be a quiver with involution as in Section \ref{sec:quiverinv}, then there is a symmetric Cartan matrix with involution $(C,\tau)$ by forgetting the orientation of $Q$, i.e.
\[
    I=Q_0,\qquad \tau=\tau|_{Q_0},\qquad c_{ij}=-\texttt{\#}\left\{h\in Q_1:\{s(h),t(h)\}=\{i,j\}\right\}\quad \textit{ for } i\neq j.
\]
Notice that the further assumptions on the quiver are equivalent to the assumptions on the Cartan matrix in \cref{sec:Yangian}. Let $\mathfrak{g}_Q$ denote the corresponding Kac--Moody algebra, with $\Lambda_i$ (resp. $\alpha_i^\vee$) being the fundamental coweights (resp. simple coroots). Fix dimension vectors $(v_i)_{i\in Q_0}$ and $(w_i)_{i\in Q_0}$, such that $v_i=v_{\tau i}$. Then we have the quantized Coulomb branch algebra
\[
    \mathcal{A}_\hbar := H^{(G_{\calO}\times F_{\calO})\rtimes\mathbb{C}_\hbar^\times}_*(\calR_{G,\bfN})
\]
associated to the following data as in Section \ref{sec:mono}:
\[
    G:=G_V^{\tau} \simeq \prod_{i\in Q_0^+}GL(V_i),\qquad \bfN = E_V^\tau \oplus L_{W,V},\qquad F=G_W.
\]
By \cref{eq:embeddiff}, we have the embedding 
\begin{equation}\label{equ:embedcou}
    \varphi:\mathcal{A}_\hbar\hookrightarrow \Diff_\hbar(T_V^\tau)\otimes H_{G_W}^*(\pt).
\end{equation}

Let 
\[
    \mu=\sum_{j\in Q_0} (w_j+w_{\tau j})\Lambda_j+\sum_{h\in Q_1^\tau}(\Lambda_{s(h)}+\Lambda_{t(h)})-\sum_{j\in Q_0}v_j\alpha_j^\vee\in P^\vee,
\]
which is clearly $\tau$-invariant. For $i\in Q_0$, we define the following elements in $\Diff_\hbar(T_V^\tau)\otimes H_{G_W}^*(\pt)$:
\begin{align}\label{equ:Biz}
    B_i(u)&=\sum_{r=1}^{v_i}\frac{1}{-u-x_{i,r}-\tfrac{\hbar}{2}} \prod_{\substack{h\in Q_1^\tau \\ s(h)=i}}(2x_{i,r}+\tfrac{3\hbar}{2})V_{\tau i}(x_{i,r}+\tfrac{\hbar}{2})\\
    &\mspace{21mu}\cdot\prod_{\substack{h\notin Q_1^\tau \\ s(h)=i}}V_{t(h)}(x_{i,r}+\tfrac{\hbar}{2})\frac{W_{\tau i}(-x_{i,r}-\tfrac{\hbar}{2})}{V_{i,r}(x_{i,r})}d_{i,r}\notag\\
    &=\sum_{r=1}^{v_i}\frac{1}{-u-x_{i,r}-\tfrac{\hbar}{2}}\prod_{\substack{h\in Q_1^\tau \\ s(h)=i}}(2x_{i,r}+\tfrac{3\hbar}{2}) \prod_{\substack{h\in Q_1 \\ s(h)=i}}V_{t(h)}(x_{i,r}+\tfrac{\hbar}{2})\frac{W_{\tau i}(-x_{i,r}-\tfrac{\hbar}{2})}{V_{i,r}(x_{i,r})}d_{i,r},
\end{align}
and
\begin{equation}\label{equ:Hu}
    \begin{aligned}
        H_i(u) &:= (-1)^{v_i-1}\bigg(\frac{2u}{(2u-\frac{\hbar}{2})(2u+\frac{\hbar}{2})}\bigg)^{c_{i,\tau i}}(-1)^{\delta_{i\to \tau(i)}}\\
        &\mspace{24mu}\cdot\frac{W_i(-u)W_{\tau i}(u)}{V_i(-u+\frac{\hbar}{2})V_i(-u-\tfrac{\hbar}{2})}\prod_{\substack{h\in Q_1 \\ s(h)=i}}V_{t(h)}(-u)\prod_{\substack{h\in Q_1 \\ s(h)=\tau i}}V_{t(h)}(u),
    \end{aligned}
\end{equation}
where $\delta_{i\to \tau(i)}$ is $1$ if there is an $h\in Q_1$ such that $s(h)=i$ and $t(h)=\tau(i)$, and is $0$ otherwise.

Expanding $B_i(u)$ and $H_i(u)$ into a Laurent series in $u^{-1}$, we let $B_{i,r}$ (resp. $H_{i,r}$) denote the coefficient of $u^{-r-1}$ of $B_i(u)$ (resp. $H_i(u)$). 
\begin{lemma}\label{lem:Hiz}
    The following holds:
    \begin{enumerate}
        \item $H_{\tau i}(u) = H_i(-u)$.
        \item $H_{i,r}\in H_{G_V^\tau\times G_W\times \mathbb{C}^*_\hbar}^*(\pt)$
        \item $H_{i,r}=0$ if $r<-\langle \alpha_i,\mu\rangle -1$, and
        \[
            H_{i,-\langle \alpha_i,\mu\rangle -1}=2^{-c_{i,\tau i}}(-1)^{v_i-1+\delta_{i\rightarrow \tau i}+w_i+\sum_{h\in Q_1, s(h)=i}v_{t(h)}}.
        \]
    \end{enumerate}
\end{lemma}
\begin{proof}
    The first one follows from the fact $x_{i,j}=-x_{\tau i, j}$, while the second one follows from the fact that the coefficients of $V_i(u)$ (resp. $W_i(u)$) are in $H_{G_V^\tau}^*(\pt)$ (resp. $H_{G_W}^*(\pt)$). By definition, the highest degree of $u$ in $H_i(u)$ is
    \[-c_{i,\tau i}+w_i+w_{\tau i}-2v_i+\sum_{\substack{h\in Q_1 \\ s(h)=i}}v_{t(h)}+\sum_{\substack{h\in Q_1 \\ s(h)=\tau i}}v_{t(h)}=\langle \alpha_i,\mu\rangle.\]
    This concludes the lemma.
\end{proof}

From now on, we fix the parameters 
\[
    \zeta_i:=\hbar^{-1}2^{-c_{i,\tau i}}(-1)^{v_i-1+\delta_{i\rightarrow \tau i}+w_i+\sum_{h\in Q_1, s(h)=i}v_{t(h)}},
\]
and consider the associated twisted shifted Yangian $\mathbf{Y}^{\tau}_{\mu}(\mathfrak{g}_Q)$. Then we have the following GKLO type representation for the shifted twisted Yangian, generalizing the results from \cite{GKLO05,KWWY14,BFN19}.
\begin{theorem}\label{thm:gklo}
    There exists a unique $H_{G_W}^*(\pt)(\hbar)$-algebra homomorphism 
    \begin{equation}\label{gklo}
        \psi\colon\mathbf{Y}^{\tau}_{\mu}(\mathfrak{g}_Q)\otimes H_{G_W}^*(\pt)\longrightarrow \Diff_\hbar(T_V^\tau)\otimes H_{G_W}^*(\pt)
    \end{equation}
    sending $h_i(z)$ to $H_i(z)$ and $b_i(z)$ to $B_i(z)$.
\end{theorem}
\begin{remark}
    The GKLO type representations for the quantum symmetric pairs are also considered in \cite{bartlett2025gklo,LWW,SSX25}.
\end{remark}

To prove the theorem, we need to check that the operators $B_i(z)$ and $H_i(z)$ satisfy all the relations in $\mathbf{Y}^{\tau}_{\mu}(\mathfrak{g}_Q)\otimes H_{G_W}^*(\pt)$. We show this in \cref{sec:relations} below.

Recall the embedding \cref{equ:embedcou}. As an immediate consequence of the above theorem, we get our main result. 
\begin{theorem}\label{thm:main}
    The algebra homomorphism $\psi$ in \cref{thm:gklo} factors through a unique $H_{G_W}^*(\pt)[\hbar,\hbar^{-1}]$-algebra homomorphism $\Psi$:
    \[
        \xymatrix{\mathbf{Y}^{\tau}_{\mu}(\mathfrak{g}_Q)\otimes H_{G_W}^*(\pt) \ar[rr]^-\Psi \ar[rd]^-\psi & &\mathcal{A}_\hbar[\hbar^{-1}] \ar@{^{(}->}[ld]^-{z^*\circ(\iota_*)^{-1}}\\
        &\Diff_\hbar(T_V^\tau)\otimes H_{G_W}^*(\pt),}
    \]
    such that for any $i\in Q_0$, $\Psi(h_{i,r})=\hbar^{-1}H_{i,r}$; for $i\in Q_0^+$,
    \[
        \Psi(b_{i,r})= \hbar^{-1}(-1)^{1+\sum_{h\notin Q_1^\tau,s(h)=i} v_{t(h)}+\sum_{h\in Q_1^\tau,s(h)=i}(v_i-1)} (-c_1(\mathcal{Q}_i)-\tfrac{\hbar}{2})^r\cap [R_{\epsilon_{i,1}}],
    \]
    and for $i\in Q_0^-$,
    \[
        \Psi(b_{i,r})= \hbar^{-1}(-1)^{1+\sum_{h\notin Q_1^\tau,s(h)=i} v_{t(h)}+\sum_{h\in Q_1^\tau,s(h)=i}(v_i-1)} (c_1(\mathcal{S}_{\tau i})+\tfrac{\hbar}{2})^r\cap [R_{-\epsilon_{\tau i,v_i}}].
    \]
\end{theorem}
\begin{proof}
    This follows from \cref{prop:fmono} and \cref{prop:fmono2}.
\end{proof}
\begin{remark}
    By \cref{lem:Hiz}(2), the image of $h_{i,r}$ lies in the Gelfand--Tsetlin subalgebra $H_{G_V^\tau\times G_W\times \mathbb{C}^*_\hbar}^*(\pt)$ of $\mathcal{A}_\hbar[\hbar^{-1}]$.  
\end{remark}

\section{Preparations for the proof of relations}\label{sec:preparations}

In this section, we prepare some notations and lemmas that will be used in the verification of relations in the following sections.

\subsection{Notations}

Let us introduce some notations. For $i\in Q_0$ and $1\leq r\leq v_i$, let
\begin{align*}
    y_{i,r}:=&\prod_{\substack{h\in Q_1^\tau \\ s(h)=i}}(2x_{i,r}+\tfrac{3\hbar}{2})V_{\tau i}(x_{i,r}+\tfrac{\hbar}{2})\prod_{\substack{h\notin Q_1^\tau \\ s(h)=i}}V_{t(h)}(x_{i,r}+\tfrac{\hbar}{2})\frac{W_{\tau i}(-x_{i,r}-\tfrac{\hbar}{2})}{V_{i,r}(x_{i,r})}d_{i,r}\\
    =&\prod_{\substack{h\in Q_1^\tau \\ s(h)=i}}(2x_{i,r}+\tfrac{3\hbar}{2})\prod_{\substack{h\in Q_1 \\ s(h)=i}}V_{t(h)}(x_{i,r}+\tfrac{\hbar}{2})\frac{W_{\tau i}(-x_{i,r}-\tfrac{\hbar}{2})}{V_{i,r}(x_{i,r})}d_{i,r}.
\end{align*}
Then
\begin{align*}
    y_{\tau i,r}=&\prod_{\substack{h\in Q_1^\tau \\ t(h)=i}}(-2x_{i,r}+\tfrac{3\hbar}{2})V_{i}(-x_{i,r}+\tfrac{\hbar}{2})\prod_{\substack{h\notin Q_1^\tau \\ s(h)=\tau i}}V_{t(h)}(-x_{i,r}+\tfrac{\hbar}{2})\frac{W_i(x_{i,r}-\tfrac{\hbar}{2})}{(-1)^{v_i-1}V_{i,r}(x_{i,r})}d_{i,r}^{-1}\\
    =&\prod_{\substack{h\in Q_1^\tau \\ t(h)=i}}(-2x_{i,r}+\tfrac{3\hbar}{2})\prod_{\substack{h\in Q_1 \\ s(h)=\tau i}}V_{t(h)}(-x_{i,r}+\tfrac{\hbar}{2})\frac{W_i(x_{i,r}-\tfrac{\hbar}{2})}{(-1)^{v_i-1}V_{i,r}(x_{i,r})}d_{i,r}^{-1},
\end{align*}
and
\[
    B_i(z) =\sum_{r=1}^{v_i}\frac{1}{-z-x_{i,r}-\tfrac{\hbar}{2}} y_{i,r}.
\]

We first prove the following two lemmas about their commutation relations.

\begin{lemma}\label{lem:yy}
    For $ i,j\in Q_0 $ and $ 1\leq r\leq v_i, 1\leq s\leq v_j $,
    \begin{equation*}
        y_{j,s}y_{i,r}=C_{i,j,r,s}y_{i,r}y_{j,s},
    \end{equation*}
    where
    \begin{equation*}
        C_{i,j,r,s}=\begin{cases}
            1, & j=i,s=r \\
            \frac{x_{i,r}-x_{i,s}+\hbar}{x_{i,r}-x_{i,s}-\hbar}, & j=i,s\neq r \\
            \text{no simple expressions}, &j=\tau i,s=r\\
            \left( \frac{x_{i,r}+x_{i,s}-\frac{\hbar}{2}}{x_{i,r}+x_{i,s}+\frac{\hbar}{2}} \right) ^{-c_{i,\tau i}}, & j=\tau i,s\neq r \\
            \left( \frac{x_{i,r}-x_{j,s}-\frac{\hbar}{2}}{x_{i,r}-x_{j,s}+\frac{\hbar}{2}} \right) ^{-c_{ij}}, & j\notin\{i,\tau i\}
        \end{cases}
    \end{equation*}
\end{lemma}
Obviously, it can also be written in a much more compact form when $ r\neq s $ or $ j\notin\{i,\tau i\} $:
\begin{equation}\label{equ:Cijrs}
    C_{i,j,r,s}=\frac{x_{i,r}-x_{j,s}+\frac{c_{ij}}{2}\hbar}{x_{i,r}-x_{j,s}-\frac{c_{ij}}{2}\hbar}.
\end{equation}

\begin{proof}
    We prove this by considering what is changed when we move $ d_{j,s} $ past $ y_{i,r} $. The first case is trivial, and we will leave the third case alone. So there are three cases:
    \begin{itemize}
        \item If $ j=i $ and $ s\neq r $, then $ d_{j,s} $ commutes with all factors in $ y_{i,r} $ except possible $ (x_{i,r}+\frac{\hbar}{2}-x_{\tau i,s}) $ (if $ \delta_{i\to\tau i}=1 $) and $ \frac{1}{(x_{i,r}-x_{i,s})} $. Thus we have
        \begin{equation*}
            d_{i,s}y_{i,r}=\bigg( \frac{x_{i,r}+x_{i,s}+\frac{3\hbar}{2}}{x_{i,r}+x_{i,s}+\frac{\hbar}{2}} \bigg) ^{\delta_{i\to\tau i}}\frac{x_{i,r}-x_{i,s}}{x_{i,r}-x_{i,s}-\hbar}y_{i,r}d_{i,s}.
        \end{equation*}
        By the exact same reason,
        \begin{equation*}
            d_{i,r}y_{i,s}=\bigg( \frac{x_{i,r}+x_{i,s}+\frac{3\hbar}{2}}{x_{i,r}+x_{i,s}+\frac{\hbar}{2}} \bigg) ^{\delta_{i\to\tau i}}\frac{x_{i,s}-x_{i,r}}{x_{i,s}-x_{i,r}-\hbar}y_{i,s}d_{i,r}.
        \end{equation*}
        So
        \begin{equation*}
            C_{i,i,r,s}=\frac{x_{i,r}-x_{i,s}+\hbar}{x_{i,r}-x_{i,s}-\hbar}.
        \end{equation*}

        \item If $ j=\tau i $ and $ s\neq r $, then $ d_{j,s} $ commutes and not commutes with the same factors as above, but with different shifts. Thus we have
        \begin{equation*}
            d_{\tau i,s}y_{i,r}=\bigg( \frac{x_{i,r}+x_{i,s}-\frac{\hbar}{2}}{x_{i,r}+x_{i,s}+\frac{\hbar}{2}} \bigg) ^{\delta_{i\to\tau i}}\frac{x_{i,r}-x_{i,s}}{x_{i,r}-x_{i,s}+\hbar}y_{i,r}d_{\tau i,s}.
        \end{equation*}
        Exchange $ i,r $ and $ \tau i,s $, we get
        \begin{equation*}
            d_{i,r}y_{\tau i,s}=\bigg( \frac{-x_{i,r}-x_{i,s}-\frac{\hbar}{2}}{-x_{i,r}-x_{i,s}+\frac{\hbar}{2}} \bigg) ^{\delta_{\tau i\to i}}\frac{x_{i,r}-x_{i,s}}{x_{i,r}-x_{i,s}+\hbar}y_{i,r}d_{\tau i,s}.
        \end{equation*}
        So
        \begin{equation*}
            C_{i,\tau i,r,s}=\bigg( \frac{x_{i,r}+x_{i,s}-\frac{\hbar}{2}}{x_{i,r}+x_{i,s}+\frac{\hbar}{2}} \bigg) ^{-c_{i,\tau i}}.
        \end{equation*}

        \item If $ j\notin\{i,\tau i\} $, then $ d_{j,s} $ commutes with all factors in $ y_{i,r} $ except possible $ (x_{i,r}+\frac{\hbar}{2}-x_{j,s}) $ (if $ \delta_{i\to j}=1 $). Thus we have
        \begin{equation*}
            d_{j,s}y_{i,r}=\bigg( \frac{x_{i,r}-x_{j,s}-\frac{\hbar}{2}}{x_{i,r}-x_{j,s}+\frac{\hbar}{2}} \bigg) ^{\delta_{i\to j}}y_{i,r}d_{j,s}.
        \end{equation*}
        Exchange $ i,r $ and $ j,s $, we get
        \begin{equation*}
            d_{i,r}y_{j,s}=\bigg( \frac{x_{j,s}-x_{i,r}-\frac{\hbar}{2}}{x_{j,s}-x_{i,r}+\frac{\hbar}{2}} \bigg) ^{\delta_{j\to i}}y_{j,s}d_{i,r}.
        \end{equation*}
        So
        \begin{equation*}
            C_{i,j,r,s}=\bigg( \frac{x_{i,r}-x_{j,s}-\frac{\hbar}{2}}{x_{i,r}-x_{j,s}+\frac{\hbar}{2}} \bigg) ^{-c_{ij}}.
        \end{equation*}
    \end{itemize}
    Thus Lemma \ref{lem:yy} is proved.
\end{proof}

The commutation relations between $ y_{j,s} $ and $ H_i(u) $ will also be useful:

\begin{lemma}\label{lem:yH}
    We have
    \begin{equation*}
        y_{j,s}H_i(u)=D_{i,j,s}(u)H_i(u)y_{j,s},
    \end{equation*}
    where
    \begin{equation*}
        D_{i,j,s}(u)=\begin{cases}
            \left( \frac{u-x_{i,s}-\hbar}{u-x_{i,s}} \right) ^{-c_{i,\tau i}}\frac{u+x_{i,s}-\frac{\hbar}{2}}{u+x_{i,s}+\frac{3\hbar}{2}}, & j=i \\
            \left( \frac{u-x_{i,s}+\hbar}{u-x_{i,s}} \right) ^{-c_{i,\tau i}}\frac{u+x_{i,s}+\frac{\hbar}{2}}{u+x_{i,s}-\frac{3\hbar}{2}}, & j=\tau i \\
            \left( \frac{u+x_{j,s}+\hbar}{u+x_{j,s}} \right) ^{-c_{ij}}\left( \frac{u-x_{j,s}-\hbar}{u-x_{j,s}} \right) ^{-c_{\tau i,j}}, & j\notin\{i,\tau i\}
        \end{cases}
    \end{equation*}
\end{lemma}

It can also be written in a unified form:
\begin{equation}\label{equ:Dijs}
    D_{i,j,s}(u)=\frac{u+x_{j,s}+(1-c_{ij})\frac{\hbar}{2}}{u+x_{j,s}+(1+c_{ij})\frac{\hbar}{2}}\cdot\frac{u-x_{j,s}-(1-c_{\tau i,j})\frac{\hbar}{2}}{u-x_{j,s}-(1+c_{\tau i,j})\frac{\hbar}{2}}.
\end{equation}

\begin{proof}
    The proof is similar to the previous one. We consider what is changed when we move $ d_{j,s} $ past $ H_i(u) $. There are three cases:
    \begin{itemize}
        \item If $ j=i $, then $ d_{j,s} $ commutes with all factors in $ H_i(u) $ except possible $ (-u-x_{\tau i,s}) $ (if $ \delta_{i\to\tau i}=1 $) and $ (u-x_{i,s}) $ (if $ \delta_{i\to\tau i}=1 $), and $ \frac{1}{(-u+\frac{\hbar}{2}-x_{i,s})(-u-\frac{\hbar}{2}-x_{i,s})} $. Thus we have
        \begin{align*}
            D_{i,i,s}(u)&=\bigg(\frac{-u+x_{i,s}+\hbar}{-u+x_{i,s}}\bigg)^{\delta_{i\to\tau i}}\bigg(\frac{u-x_{i,s}-\hbar}{u-x_{i,s}}\bigg)^{\delta_{\tau i\to i}}\frac{-u+\frac{\hbar}{2}-x_{i,s}}{-u-\frac{3\hbar}{2}-x_{i,s}}\\
            &=\bigg(\frac{u-x_{i,s}-\hbar}{u-x_{i,s}}\bigg)^{-c_{i,\tau i}}\frac{u+x_{i,s}-\frac{\hbar}{2}}{u+x_{i,s}+\frac{3\hbar}{2}}.
        \end{align*}

        \item If $ j=\tau i $, then using $ H_i(u)=H_{\tau i}(-u) $, we have
        \begin{equation*}
            D_{i,\tau i,r}(u)=\bigg(\frac{u-x_{i,s}+\hbar}{u-x_{i,s}}\bigg)^{-c_{i,\tau i}}\frac{u+x_{i,s}+\frac{\hbar}{2}}{u+x_{i,s}-\frac{3\hbar}{2}}.
        \end{equation*}

        \item If $ j\notin\{i,\tau i\} $, then $ d_{j,s} $ commutes with all factors in $ H_i(u) $ except possible $ (-u-x_{j,s}) $ (if $ \delta_{i\to j}=1 $), $ (-u-x_{\tau j,s}) $ (if $ \delta_{i\to\tau j}=1 $), $ (u-x_{j,s}) $ (if $ \delta_{\tau i\to j}=1 $), and $ (u-x_{\tau j,s}) $ (if $ \delta_{\tau i\to\tau j}=1 $). Thus we have
        \begin{align*}
            &\mspace{24mu}D_{i,j,s}(u)\\
            &=\bigg(\frac{-u-x_{j,s}-\hbar}{-u-x_{j,s}}\bigg)^{\delta_{i\to j}}\bigg(\frac{-u+x_{j,s}+\hbar}{-u+x_{j,s}}\bigg)^{\delta_{i\to\tau j}}\bigg(\frac{u-x_{j,s}-\hbar}{u-x_{j,s}}\bigg)^{\delta_{\tau i\to j}}\bigg(\frac{u+x_{j,s}+\hbar}{u+x_{j,s}}\bigg)^{\delta_{\tau i\to\tau j}}\\
            &=\bigg(\frac{u+x_{j,s}+\hbar}{u+x_{j,s}}\bigg)^{-c_{ij}}\bigg(\frac{u-x_{j,s}-\hbar}{u-x_{j,s}}\bigg)^{-c_{\tau i,j}}.
        \end{align*}
    \end{itemize}
    Thus Lemma \ref{lem:yH} is proved.
\end{proof}

\subsection{Expressions of $ (H_i(z))^\circ $}

We need some combinatorial identities.
For a Laurant series $f(z)=\sum_{i\in \mathbb{Z}} a_iz^{i}$ in $z^{-1}$,
we denote the truncation 
\[\big(f(z)\big)^{\circ}
:=\sum_{i< 0}a_iz^i.\] 
By definition,
\begin{equation}\label{equ:trundivbyz}
    u\big(f(u)\big)^\circ+v\big(f(-v)\big)^\circ=\sum_{i<0}a_iu^{i+1}-\sum_{i<0}a_i(-v)^{i+1}=\big(uf(u)\big)^{\circ}-\big((-v)f(-v)\big)^{\circ}.
\end{equation}

We will use frequently the following simple lemma. 
\begin{lemma}\cite[Lemma 5.1]{SSX25}\label{lem:partialfrac}
    Assume $f(z)=\dfrac{g(z)}{\prod_{i=1}^n (z-z_i)}$, where $z_1,\ldots,z_n$ are distinct and $g(z)\in \mathbb{C}[z]$. Then
\[(f(z))^\circ=\sum_{i=1}^n\frac{1}{z-z_i}\frac{g(z_i)}{\prod_{j\neq i}(z_i-z_j)}=\sum_{i=1}^n\frac{\operatorname{Res}_{z=z_i}(f(z)dz)}{z-z_i}.\]
\end{lemma}

\begin{lemma}\label{lem:resH}
    The function $ H_i(u) $ has simple poles at $ -x_{i,r}\pm\frac{\hbar}{2}(1\leq r\leq v_i) $ and at $ 0 $ if $ c_{i,\tau i}=-1 $, and its residues are given by
    \begin{gather*}
        \operatorname{Res}_{u=-x_{i,r}-\frac{\hbar}{2}}(H_i(u))=-\frac{1}{\hbar}\bigg(\frac{-2x_{i,r}-\tfrac{\hbar}{2}}{-2x_{i,r}-\hbar}\bigg)^{-c_{i,\tau i}}y_{i,r}y_{\tau i,r},\\
        \operatorname{Res}_{u=-x_{i,r}+\frac{\hbar}{2}}(H_i(u))=\frac{1}{\hbar}\bigg(\frac{-2x_{i,r}+\frac{\hbar}{2}}{-2x_{i,r}+\hbar}\bigg)^{-c_{i,\tau i}}y_{\tau i,r}y_{i,r}.
    \end{gather*}
\end{lemma}

\begin{proof}
    We define auxiliary functions
    \begin{equation*}
        H_{i\setminus r}(u):=(-u+\tfrac{\hbar}{2}-x_{i,r})(-u-\tfrac{\hbar}{2}-x_{i,r})H_i(u).
    \end{equation*}

    Direct computation shows that
    \begin{align*}
        d_{i,r}y_{\tau i,r}&=d_{i,r}(-2x_{i,r}+\tfrac{3\hbar}{2})^{\delta_{\tau i\to i}}\prod_{\substack{h\in Q_1 \\ s(h)=\tau i}}V_{t(h)}(-x_{i,r}+\tfrac{\hbar}{2})\cdot\frac{W_{i}(x_{i,r}-\frac{\hbar}{2})}{V_{\tau i,r}(-x_{i,r})}d_{i,r}^{-1}\\
        &=(-2x_{i,r}-\tfrac{\hbar}{2})^{\delta_{\tau i\to i}}\prod_{\substack{h\in Q_1 \\ s(h)=\tau i}}V_{t(h)}(-x_{i,r}-\tfrac{\hbar}{2})\cdot\bigg(\frac{-2x_{i,r}-\frac{3\hbar}{2}}{-2x_{i,r}-\frac{\hbar}{2}}\bigg)^{\delta_{\tau i\to i}}\cdot\frac{W_{i}(x_{i,r}+\frac{\hbar}{2})}{V_{\tau i,r}(-x_{i,r}-\hbar)}\\
        &=(-1)^{v_i-1}(-2x_{i,r}-\tfrac{3\hbar}{2})^{\delta_{\tau i\to i}}\prod_{\substack{h\in Q_1 \\ s(h)=\tau i}}V_{t(h)}(-x_{i,r}-\tfrac{\hbar}{2})\cdot\frac{W_{i}(x_{i,r}+\frac{\hbar}{2})}{V_{i,r}(x_{i,r}+\hbar)}
    \end{align*}
    By multiplying the function $ y_{i,r}d_{i,r}^{-1} $ from the left, we get
    \begin{align*}
        y_{i,r}y_{\tau i,r}&=(-1)^{v_i-1}(2x_{i,r}+\tfrac{3\hbar}{2})^{\delta_{i\to\tau i}}(-2x_{i,r}-\tfrac{3\hbar}{2})^{\delta_{\tau i\to i}}\\
        &\mspace{21mu}\cdot\prod_{\substack{h\in Q_1 \\ s(h)=i}}V_{t(h)}(x_{i,r}+\tfrac{\hbar}{2})\prod_{\substack{h\in Q_1 \\ s(h)=\tau i}}V_{t(h)}(-x_{i,r}-\tfrac{\hbar}{2})\cdot\frac{W_{\tau i}(-x_{i,r}-\frac{\hbar}{2})W_{i}( x_{i,r}+\frac{\hbar}{2} ) }{V_{i,r}(x_{i,r})V_{i,r}(x_{i,r}+\hbar)}\\
        &=(-1)^{v_i-1+\delta_{i\to\tau i}}(-2x_{i,r}-\tfrac{3\hbar}{2})^{-c_{i,\tau i}}\\
        &\mspace{21mu}\cdot\prod_{\substack{h\in Q_1 \\ s(h)=i}}V_{t(h)}(x_{i,r}+\tfrac{\hbar}{2})\prod_{\substack{h\in Q_1 \\ s(h)=\tau i}}V_{t(h)}(-x_{i,r}-\tfrac{\hbar}{2})\cdot\frac{W_{\tau i}(-x_{i,r}-\frac{\hbar}{2})W_{i}( x_{i,r}+\frac{\hbar}{2} ) }{V_{i,r}(x_{i,r})V_{i,r}(x_{i,r}+\hbar)}.
    \end{align*}

    While by definition of $ H_{i\setminus r}(u) $, we have
    \begin{align*}
        &\mspace{24mu}H_{i\setminus r}(-x_{i,r}-\tfrac{\hbar}{2})\\
        &=(-1)^{v_i-1+\delta_{i\to\tau i}}\bigg(\frac{-2x_{i,r}-\hbar}{(-2x_{i,r}-\frac{3\hbar}{2})(-2x_{i,r}-\frac{\hbar}{2})}\bigg)^{c_{i,\tau i}}\\
        &\mspace{21mu}\cdot\prod_{\substack{h\in Q_1 \\ s(h)=i}}V_{t(h)}(x_{i,r}+\tfrac{\hbar}{2})\prod_{\substack{h\in Q_1 \\ s(h)=\tau i}}V_{t(h)}(-x_{i,r}-\tfrac{\hbar}{2})\cdot\frac{W_i(x_{i,r}+\tfrac{\hbar}{2})W_{\tau i}(-x_{i,r}-\tfrac{\hbar}{2})}{V_{i,r}(x_{i,r}+\hbar)V_{i,r}(x_{i,r})}\\
        &=\bigg(\frac{-2x_{i,r}-\frac{\hbar}{2}}{-2x_{i,r}-\hbar}\bigg)^{-c_{i,\tau i}}y_{i,r}y_{\tau i,r}.
    \end{align*}
    so
    \begin{align*}
        \operatorname{Res}_{u=-x_{i,r}-\frac{\hbar}{2}}(H_i(u))&=\left.\frac{H_{i\setminus r}(u)}{u+x_{i,r}-\frac{\hbar}{2}}\right|_{u=-x_{i,r}-\frac{\hbar}{2}}\\
        &=\frac{H_{i\setminus r}(-x_{i,r}-\frac{\hbar}{2})}{-\hbar}\\
        &=-\frac{1}{\hbar}\bigg(\frac{-2x_{i,r}-\frac{\hbar}{2}}{-2x_{i,r}-\hbar}\bigg)^{-c_{i,\tau i}}y_{i,r}y_{\tau i,r}.
    \end{align*}
    Exchange $ i $ and $ \tau i $, we get
    \begin{align*}
        \operatorname{Res}_{u=-x_{i,r}+\frac{\hbar}{2}}(H_i(u))&=\operatorname{Res}_{u=x_{\tau i,r}+\frac{\hbar}{2}}(H_{\tau i}(-u))\\
        &=-\operatorname{Res}_{u=-x_{\tau i,r}-\frac{\hbar}{2}}(H_{\tau i}(u))\\
        &=\frac{1}{\hbar}\bigg(\frac{-2x_{i,r}+\frac{\hbar}{2}}{-2x_{i,r}+\hbar}\bigg)^{-c_{i,\tau i}}y_{\tau i,r}y_{i,r}.
    \end{align*}
\end{proof}
Combining Lemma \ref{lem:partialfrac} and Lemma \ref{lem:resH}, we get the following corollary.
\begin{corol}\label{cor:cequal0H}
    We have
    \begin{equation*}
        \hbar (H_i(u))^\circ=\sum_{r=1}^{v_i}\frac{1}{u+x_{i,r}-\frac{\hbar}{2}}y_{\tau i,r}y_{i,r}-\sum_{r=1}^{v_i}\frac{1}{u+x_{i,r}+\frac{\hbar}{2}}y_{i,r}y_{\tau i,r}
    \end{equation*}
    for $ c_{i,\tau i}=0 $, and
    \begin{equation*}
        (2u\hbar H_i(u))^\circ=\sum_{r=1}^{v_i}\frac{-2x_{i,r}+\frac{\hbar}{2}}{u+x_{i,r}-\frac{\hbar}{2}}y_{\tau i,r}y_{i,r}-\sum_{r=1}^{v_i}\frac{-2x_{i,r}-\frac{\hbar}{2}}{u+x_{i,r}+\frac{\hbar}{2}}y_{i,r}y_{\tau i,r}
    \end{equation*}
    for $ c_{i,\tau i}=-1 $. One can also write them in a unified way:
    \begin{equation*}
        (2u\hbar H_i(u))^\circ=\sum_{r=1}^{v_i}\frac{-2x_{i,r}+\hbar+\frac{c_{i,\tau i}\hbar}{2}}{u+x_{i,r}-\frac{\hbar}{2}}y_{\tau i,r}y_{i,r}-\sum_{r=1}^{v_i}\frac{-2x_{i,r}-\hbar-\frac{c_{i,\tau i}\hbar}{2}}{u+x_{i,r}+\frac{\hbar}{2}}y_{i,r}y_{\tau i,r}.
    \end{equation*}
\end{corol}

The next lemma is about the value $ H_i(-x_{i,r}-\frac{3\hbar}{2}) $ when $ c_{i,\tau i}=-1 $.

\begin{lemma}\label{lem:Hiat3h/2}
    If $ c_{i,\tau i}=-1 $, then
    \begin{equation*}
        H_i(-x_{i,r}-\tfrac{3\hbar}{2})=\frac{1}{2\hbar^2}\frac{2x_{i,r}+\frac{3\hbar}{2}}{2x_{i,r}+3\hbar}\cdot d_{i,r}y_{i,r}y_{\tau i,r}d_{i,r}^{-1}.
    \end{equation*}
\end{lemma}

\begin{proof}
    We calculate $ d_{i,r}^{-1}H_i(-x_{i,r}-\frac{3\hbar}{2})d_{i,r} $ first:
    \begin{align*}
        &\mspace{24mu}d_{i,r}^{-1}H_i(-x_{i,r}-\tfrac{3\hbar}{2})d_{i,r}\\
        &=d_{i,r}^{-1}(-1)^{v_i-1+\delta_{i\to\tau i}}\frac{(-2x_{i,r}-\frac{7\hbar}{2})(-2x_{i,r}-\frac{5\hbar}{2})}{-2x_{i,r}-3\hbar}\\
        &\mspace{21mu}\cdot\prod_{\substack{h\in Q_1 \\ s(h)=i}}V_{t(h)}(x_{i,r}+\tfrac{3\hbar}{2})\prod_{\substack{h\in Q_1 \\ s(h)=\tau i}}V_{t(h)}(-x_{i,r}-\tfrac{3\hbar}{2})\cdot\frac{W_i(x_{i,r}+\frac{3\hbar}{2})W_{\tau i}(-x_{i,r}-\frac{3\hbar}{2})}{V_i(x_{i,r}+\hbar)V_i(x_{i,r}+2\hbar)}d_{i,r}\\
        &=(-1)^{v_i-1+\delta_{i\to\tau i}}\frac{(-2x_{i,r}-\frac{3\hbar}{2})(-2x_{i,r}-\frac{\hbar}{2})}{-2x_{i,r}-\hbar}\\
        &\mspace{21mu}\cdot d_{i,r}^{-1}\prod_{\substack{h\in Q_1 \\ s(h)=i}}V_{t(h)}(x_{i,r}+\tfrac{3\hbar}{2})\prod_{\substack{h\in Q_1 \\ s(h)=\tau i}}V_{t(h)}(-x_{i,r}-\tfrac{3\hbar}{2})\cdot\frac{W_i(x_{i,r}+\frac{3\hbar}{2})W_{\tau i}(-x_{i,r}-\frac{3\hbar}{2})}{V_i(x_{i,r}+\hbar)V_i(x_{i,r}+2\hbar)}d_{i,r}\\
        &=(-1)^{v_i-1+\delta_{i\to\tau i}}\frac{(-2x_{i,r}-\frac{3\hbar}{2})(-2x_{i,r}-\frac{\hbar}{2})}{-2x_{i,r}-\hbar}\cdot\frac{2x_{i,r}-\frac{\hbar}{2}}{2x_{i,r}+\frac{\hbar}{2}}\\
        &\mspace{21mu}\cdot\prod_{\substack{h\in Q_1 \\ s(h)=i}}V_{t(h)}(x_{i,r}+\tfrac{\hbar}{2})\prod_{\substack{h\in Q_1 \\ s(h)=\tau i}}V_{t(h)}(-x_{i,r}-\tfrac{\hbar}{2})\cdot d_{i,r}^{-1}\frac{W_i(x_{i,r}+\frac{3\hbar}{2})W_{\tau i}(-x_{i,r}-\frac{3\hbar}{2})}{2\hbar^2V_{i,r}(x_{i,r}+\hbar)V_{i,r}(x_{i,r}+2\hbar)}d_{i,r}\\
        &=(-1)^{v_i-1+\delta_{i\to\tau i}}\frac{(-2x_{i,r}-\frac{3\hbar}{2})(-2x_{i,r}+\frac{\hbar}{2})}{-2x_{i,r}-\hbar}\\
        &\mspace{21mu}\cdot\prod_{\substack{h\in Q_1 \\ s(h)=i}}V_{t(h)}(x_{i,r}+\tfrac{\hbar}{2})\prod_{\substack{h\in Q_1 \\ s(h)=\tau i}}V_{t(h)}(-x_{i,r}-\tfrac{\hbar}{2})\cdot\frac{W_i(x_{i,r}+\frac{\hbar}{2})W_{\tau i}(-x_{i,r}-\frac{\hbar}{2})}{2\hbar^2V_{i,r}(x_{i,r})V_{i,r}(x_{i,r}+\hbar)}\\
        &=\frac{-2x_{i,r}+\frac{\hbar}{2}}{-2x_{i,r}-\hbar}\cdot\frac{1}{2\hbar^2}y_{i,r}y_{\tau i,r},
    \end{align*}
    so
    \begin{align*}
        H_i(-x_{i,r}-\tfrac{3\hbar}{2})&=d_{i,r}\frac{-2x_{i,r}+\frac{\hbar}{2}}{-2x_{i,r}-\hbar}\cdot\frac{1}{2\hbar^2}y_{i,r}y_{\tau i,r}d_{i,r}^{-1}\\
        &=\frac{1}{2\hbar^2}\frac{2x_{i,r}+\frac{3\hbar}{2}}{2x_{i,r}+3\hbar}\cdot d_{i,r}y_{i,r}y_{\tau i,r}d_{i,r}^{-1}.
    \end{align*}
\end{proof}

\subsection{Serre-type relations for $ y_{i,r} $}
Assume $ c_{i,\tau i}=-1 $. We will need the following two lemmas about Serre-type relations for $ y_{i,r} $'s:

\begin{lemma}\label{lem:yserre1}
    For distinct $ r_1,r_2,s $, we have
    \begin{equation*}
        \mathrm{Sym}_{r_1,r_2}(y_{i,r_1}y_{i,r_2}y_{\tau i,s}-2y_{i,r_1}y_{\tau i,s}y_{i,r_2}+y_{\tau i,s}y_{i,r_1}y_{i,r_2})=0.
    \end{equation*}
\end{lemma}
\begin{proof}
    By Lemma \ref{lem:yy}, we get
    \begin{align*}
        &\mspace{24mu}y_{i,r_1}y_{i,r_2}y_{\tau i,s}-2y_{i,r_1}y_{\tau i,s}y_{i,r_2}+y_{\tau i,s}y_{i,r_1}y_{i,r_2}\\
        &=(1-2C_{i,\tau i,r_2,s}+C_{i,\tau i,r_1,s}C_{i,\tau i,r_2,s})y_{i,r_1}y_{i,r_2}y_{\tau i,s}
    \end{align*}
    and
    \begin{align*}
        &\mspace{24mu}y_{i,r_2}y_{i,r_1}y_{\tau i,s}-2y_{i,r_2}y_{\tau i,s}y_{i,r_1}+y_{\tau i,s}y_{i,r_2}y_{i,r_1}\\
        &=(1-2C_{i,\tau i,r_1,s}+C_{i,\tau i,r_1,s}C_{i,\tau i,r_2,s})y_{i,r_2}y_{i,r_1}y_{\tau i,s}\\
        &=(1-2C_{i,\tau i,r_1,s}+C_{i,\tau i,r_1,s}C_{i,\tau i,r_2,s})C_{i,i,r_1,r_2}y_{i,r_1}y_{i,r_2}y_{\tau i,s}.
    \end{align*}
    One can check that
    \begin{equation*}
        (1-2C_{i,\tau i,r_2,s}+C_{i,\tau i,r_1,s}C_{i,\tau i,r_2,s})+(1-2C_{i,\tau i,r_1,s}+C_{i,\tau i,r_1,s}C_{i,\tau i,r_2,s})C_{i,i,r_1,r_2}=0.
    \end{equation*}
\end{proof}


\begin{lemma}\label{lem:yserre2}
    For distinct $ r,s $, we have
    \begin{equation*}
        y_{i,r}y_{i,r}y_{\tau i,s}-2y_{i,r}y_{\tau i,s}y_{i,r}+y_{\tau i,s}y_{i,r}y_{i,r}=0.
    \end{equation*}
\end{lemma}
\begin{proof}
    We have
    \begin{align*}
        &\mspace{24mu}y_{i,r}y_{i,r}y_{\tau i,s}-2y_{i,r}y_{\tau i,s}y_{i,r}+y_{\tau i,s}y_{i,r}y_{i,r}\\
        &=\big( 1-2d_{i,r}C_{i,\tau i,r,s}d_{i,r}^{-1}+C_{i,\tau i,r,s}(d_{i,r}^{-1}C_{i,i,r,s}d_{i,r}) \big) y_{i,r}y_{i,r}y_{\tau i,s}\\
        &=\left( 1-2\frac{x_{i,r}+x_{i,s}+\frac{\hbar}{2}}{x_{i,r}+x_{i,s}+\frac{3\hbar}{2}}+\frac{x_{i,r}+x_{i,s}-\frac{\hbar}{2}}{x_{i,r}+x_{i,s}+\frac{\hbar}{2}}\frac{x_{i,r}+x_{i,s}+\frac{\hbar}{2}}{x_{i,r}+x_{i,s}+\frac{3\hbar}{2}} \right) y_{i,r}y_{i,r}y_{\tau i,s}.
    \end{align*}
    One can check that
    \begin{equation*}
        1-2\frac{x_{i,r}+x_{i,s}+\frac{\hbar}{2}}{x_{i,r}+x_{i,s}+\frac{3\hbar}{2}}+\frac{x_{i,r}+x_{i,s}-\frac{\hbar}{2}}{x_{i,r}+x_{i,s}+\frac{\hbar}{2}}\frac{x_{i,r}+x_{i,s}+\frac{\hbar}{2}}{x_{i,r}+x_{i,s}+\frac{3\hbar}{2}}=0.
    \end{equation*}
\end{proof}


\section{Relations}\label{sec:relations}
In this section, we start to prove Theorem \ref{thm:gklo} by checking that the operators $B_i(z)$ and $H_i(z)$ satisfy the relations in the shifted twisted Yangian. 

The relation $\left[H_i(u), H_j(v)\right]=0$ is obvious, while $H_{\tau i}(u)=H_i(-u)$ is proved in \cref{lem:Hiz}(1). 

\subsection{Relation \eqref{eq:hbcomm}: commutativity of $h_{i,r}$ and $b_{j,s}$}
In this section, we check the relation \eqref{eq:hbcomm(z)}
\begin{align}\label{equ:hibj}
&
\left(u^2-v^2-\frac{c_{i j} c_{\tau i, j}}{4} \hbar^2\right)\left[h_i(u), b_j(v)\right]-\bigg(\frac{c_{i j}-c_{\tau i, j}}{2} \hbar u +\frac{c_{i j}+c_{\tau i, j}}{2} \hbar v\bigg)\left\{h_i(u), b_j(v)\right\}\\
&\notag
+\hbar\left[h_i(u), b_{j, 1}\right] 
+\hbar v\left[h_i(u), b_{j, 0}\right]+\frac{c_{i j}+c_{\tau i, j}}{2} \hbar^2\left\{h_i(u), b_{j, 0}\right\}=0.
\end{align}

It suffices to prove coefficients of $ H_i(u)y_{j,s} $ are equal to zero for every $ i,j,r $. We can express this in $ D_{i,j,s}(u) $:
\begin{align*}
    &\left( u^2-v^2-\frac{c_{ij}c_{\tau i,j}}{4}\hbar^2 \right) \frac{1-D_{i,j,s}(u)}{-v-x_{j,s}-\frac{\hbar}{2}}-\bigg(\frac{c_{i j}-c_{\tau i, j}}{2} \hbar u +\frac{c_{i j}+c_{\tau i, j}}{2} \hbar v\bigg)\frac{1+D_{i,j,s}(u)}{-v-x_{j,s}-\frac{\hbar}{2}}\\
    &+(x_{j,s}+\tfrac{\hbar}{2})(1-D_{i,j,s}(u))-v(1-D_{i,j,s}(u))-\frac{c_{ij}+c_{i,\tau j}}{2}\hbar(1+D_{i,j,s}(u))=0.
\end{align*}
Substitute the expression of $ D_{i,j,s}(u) $ from (\ref{equ:Dijs}) in it, we get an algebraic identity. This finishes the proof of \eqref{eq:hbcomm}.


\subsection{Relation \eqref{eq:bcomm}: commutativity of $b_{i,r}$'s}
In this section, we check the relation \eqref{eq:bcomm(z)}
\begin{align}\label{equ:bbij}
    (u-v)\left[b_i(u), b_j(v)\right]-&\frac{c_{i j}}{2} \hbar\left\{b_i(u), b_j(v)\right\}-\hbar\left(\left[b_{i, 0}, b_j(v)\right]-\left[b_i(u), b_{j, 0}\right]\right) \\\notag&\qquad =-\delta_{\tau i, j} \hbar\left(\frac{2 u}{u+v} h_i^{\circ}(u)+\frac{2 v}{u+v} h_j^{\circ}(v)\right).
\end{align}
Recall that by our convention, $\hbar b_{i,m}$ maps to $B_{i,m}$, the coefficient of $u^{-m-1}$ in $B_i(u)$.

We split the proof into the following cases:
\begin{itemize}
    \item $j\notin\{i,\tau i\}$;
    \item $j=i$;
    \item $j=\tau i$.
\end{itemize}

\subsubsection{The case $j\notin\{i,\tau i\}$}
In this case, $ y_{i,r} $ commutes with $ x_{j,s} $, so the coefficient of $ y_{i,r}y_{j,s} $ in $ B_i(u)B_j(v) $ is simply the product of coefficients of $ y_{i,r} $ and $ y_{j,s} $ in each function. We replace every $ y_{j,s}y_{i,r} $ with $ C_{i,j,r,s}y_{i,r}y_{j,s} $, and then the coefficient of $ y_{i,r}y_{j,s} $ in left hand side of \eqref{equ:bbij} is
\begin{align*}
    &(u-v)\frac{1}{-u-x_{i,r}-\frac{\hbar}{2}}\frac{1}{-v-x_{j,s}-\frac{\hbar}{2}}(1-C_{i,j,r,s})\\
    &-\frac{c_{ij}}{2}\hbar\frac{1}{-u-x_{i,r}-\frac{\hbar}{2}}\frac{1}{-v-x_{i,s}-\frac{\hbar}{2}}(1+C_{i,j,r,s})+\frac{1-C_{i,j,r,s}}{-v-x_{i,s}-\frac{\hbar}{2}}-\frac{1-C_{i,j,r,s}}{-u-x_{i,r}-\frac{\hbar}{2}},
\end{align*}
which simplifies to zero by the definition of $ C_{i,j,r,s} $ in \eqref{equ:Cijrs}.


\subsubsection{The case $j=i$}

For terms with $ r\neq s $, the above argument shows that the coefficient of $ y_{i,r}y_{i,s} $, where $ y_{i,r} $ comes from $ B_i(u) $ and $ y_{i,s} $ comes from $ B_i(v) $, is zero. The other half with $ y_{i,s} $ from $ B_i(u) $ and $ y_{i,r} $ from $ B_i(v) $ is also zero by symmetry.

For terms with $ r=s $, the coefficient of $ y_{i,r}y_{i,r} $ in the left hand side is
\begin{align*}
    &(u-v)\frac{1}{-u-x_{i,r}-\frac{\hbar}{2}}\frac{1}{-v-x_{i,r}-\frac{3\hbar}{2}}-(u-v)\frac{1}{-u-x_{i,r}-\frac{3\hbar}{2}}\frac{1}{-v-x_{i,r}-\frac{\hbar}{2}}\\
    &-\hbar\frac{1}{-u-x_{i,r}-\frac{\hbar}{2}}\frac{1}{-v-x_{i,r}-\frac{3\hbar}{2}}-\hbar\frac{1}{-u-x_{i,r}-\frac{3\hbar}{2}}\frac{1}{-v-x_{i,r}-\frac{\hbar}{2}}\\
    &+\frac{1}{-v-x_{i,r}-\frac{3\hbar}{2}}-\frac{1}{-v-x_{i,r}-\frac{\hbar}{2}}-\frac{1}{-u-x_{i,r}-\frac{\hbar}{2}}+\frac{1}{-u-x_{i,r}-\frac{3\hbar}{2}}.
\end{align*}
This also simplifies to zero.


\subsubsection{The case $j=\tau i$}\label{sec:BBitaui}

The terms with $ r\neq s $ cancel out by the same identity as in the case $ j\notin\{i,\tau i\} $. For terms with $ r=s $, there is no commutation between $ y_{i,r} $ and $ y_{\tau i,r} $, so we treat them separately. The coefficient of $ y_{i,r}y_{\tau i,r} $ in left side is
\begin{align*}
    &(u-v)\frac{1}{-u-x_{i,r}-\frac{\hbar}{2}}\frac{1}{-v+x_{i,r}+\frac{\hbar}{2}}\\
    &-\frac{c_{i,\tau i}}{2}\hbar\frac{1}{-u-x_{i,r}-\frac{\hbar}{2}}\frac{1}{-v+x_{i,r}+\frac{\hbar}{2}}+\frac{1}{-v+x_{i,r}+\frac{\hbar}{2}}-\frac{1}{-u-x_{i,r}-\frac{\hbar}{2}}.
\end{align*}

By (\ref{equ:trundivbyz}), the right hand side is
\begin{equation*}
    -\frac{\hbar}{u+v}(2u(H_i(u))^\circ+2v(H_{\tau i}(v))^\circ)=-\frac{1}{u+v}((2u\hbar H_i(u))^\circ-(-2v\hbar H_i(-v))^\circ),
\end{equation*}
and by Corollary \ref{cor:cequal0H}, the coefficient of $ y_{i,r}y_{\tau i,r} $ in it is
\begin{equation*}
    -\frac{1}{u+v}\bigg(\frac{2x_{i,r}+\hbar+\frac{c_{i,\tau i}\hbar}{2}}{u+x_{i,r}+\frac{\hbar}{2}}-\frac{2x_{i,r}+\hbar+\frac{c_{i,\tau i}\hbar}{2}}{-v+x_{i,r}+\frac{\hbar}{2}}\bigg).
\end{equation*}
One can check they are equal. The coefficient of $ y_{\tau i,r}y_{i,r} $ in both sides can be checked similarly, or by exchanging $ i $ and $ \tau i $, $ u $ and $ v $. This finishes the proof of \eqref{eq:bcomm}.


\subsection{Relation \eqref{eq:commSerre(z)}}\label{sec:BBijc=0}
In this paragraph, we check the following relation 
\[(u+v)\left[b_i(u), b_j(v)\right]=\delta_{\tau i, j}\hbar\left(h_j^\circ(v)-h_i^\circ(u)\right), \qquad c_{ij}=0.\]
There are two cases depending on whether $j$ equals to $\tau i$ or not.

\subsubsection{The case $j\neq \tau i$} 
Since $c_{ij}=0$, by Lemma \ref{lem:yy}, there is always $ [y_{i,r},y_{j,s}]=0 $, so it is obvious to see that 
\[[B_i(u),B_j(v)]=0.\]

\subsubsection{The case $j=\tau i$.}
There is still $[y_{i,r}, y_{\tau i,s}]=0$ for $ r\neq s $. Therefore,
\begin{align*}
(u+v)[B_i(u),B_{\tau i}(v)] 
    =\,&\bigg(\sum_{r=1}^{v_i} \frac{1}{-v+x_{i,r}-\tfrac{\hbar}{2}}y_{\tau i,r}y_{i,r}-\sum_{r=1}^{v_i}\frac{1}{-v+x_{i,r}+\tfrac{\hbar}{2}}y_{i,r}y_{\tau i,r}\bigg)\\
    &-\bigg(\sum_{r=1}^{v_i} \frac{1}{u+x_{i,r}-\tfrac{\hbar}{2}}y_{\tau i,r}y_{i,r}-\sum_{r=1}^{v_i}\frac{1}{u+x_{i,r}+\tfrac{\hbar}{2}}y_{i,r}y_{\tau i,r}\bigg)\\
    =\,&\hbar\big(H_i^\circ(-v)-H^\circ_i(u)\big)
    =\hbar \big(H^\circ_{\tau i}(v)-H^\circ_i(u)\big).
\end{align*}
Here the last equality follows from Corollary \ref{cor:cequal0H}. This finishes the proof of relation \eqref{eq:commSerre(z)}.

\subsection{Serre Relations}
\label{iSerre(z)}
The Serre Relation \eqref{eq:usualSerre(z)} can be checked exactly the same as in \cite[Appendix B(vi)]{BFN19}, so we omit it.

Now let us check the $\imath$Serre Relation \eqref{eq:iSerre}. In \cite{SSX25}, the authors make use of Bernoulli polynomials and prove the following trick:
\begin{prop}\cite[Proposition 6.3]{SSX25}
\label{prop:finite=>affine}
Suppose that \eqref{eq:hcomm}--\eqref{eq:bcomm} are preserved under \cref{gklo}, then \eqref{eq:iSerre} for arbitrary $k_1,k_2,r\geq 0$ follows from its special case when $k_1=k_2=r=0$. 
\end{prop}
Hence, it suffices to check Relation \eqref{eq:iSerre} at $k_1=k_2=0$, which admits a generating function 
\begin{equation}\label{equ:iSerreitaui}
	\hbar^2[b_{i,0},[b_{i,0},b_{\tau(i)}(v)]]=\left(4v
	\hbar [b_i(3v), h_{\tau(i)}(v)]\right)^\circ.
\end{equation}
Recall that $\hbar b_{i,0}$ is sent to $B_{i,0}=-\sum_{r=1}^{v_i}y_{i,r}d_{i,r}$.
By direct computation, we get
\begin{align*}
    B_{i,0}^2B_{\tau i}(v)
    &=\sum_{s}\frac{1}{-v+x_{i,s}+\frac{3\hbar}{2}}y_{i,s}y_{i,s}y_{\tau i,s}+\sum_{r\neq s}\frac{1+C_{i,i,r,s}}{-v+x_{i,s}+\frac{\hbar}{2}}y_{i,r}y_{i,s}y_{\tau i,s}\\
    &\mspace{21mu}+\sum_{r\neq s}\frac{1}{-v+x_{i,s}-\frac{\hbar}{2}}y_{i,r}y_{i,r}y_{\tau i,s}+\sum_{r_i\neq s,r_1\neq r_2}\frac{1}{-v+x_{i,s}-\frac{\hbar}{2}}y_{i,r_1}y_{i,r_2}y_{\tau i,s},
\end{align*}
\begin{align*}
    B_{i,0}B_{\tau i}(v)B_{i,0}
    &=\sum_{s}\frac{1}{-v+x_{i,s}+\frac{\hbar}{2}}y_{i,s}y_{\tau i,s}y_{i,s}\\
    &\mspace{21mu}+\sum_{r\neq s}\frac{1}{-v+x_{i,s}-\frac{\hbar}{2}}y_{i,r}y_{\tau i,s}y_{i,s}+\sum_{r\neq s}\frac{1}{-v+x_{i,s}+\frac{\hbar}{2}}y_{i,s}y_{\tau i,s}y_{i,r}\\
    &\mspace{21mu}+\sum_{r\neq s}\frac{1}{-v+x_{i,s}-\frac{\hbar}{2}}y_{i,r}y_{\tau i,s}y_{i,r}+\sum_{r_i\neq s,r_1\neq r_2}\frac{1}{-v+x_{i,s}-\frac{\hbar}{2}}y_{i,r_1}y_{\tau i,s}y_{i,r_2},
\end{align*}
and
\begin{align*}
    B_{\tau i}(v)B_{i,0}^2
    &=\sum_{r_1,r_2,s=1}^{v_i}\frac{1}{-v+x_{i,s}-\frac{\hbar}{2}}y_{\tau i,s}y_{i,r_1}y_{i,r_2}\\
    &=\sum_{s}\frac{1}{-v+x_{i,s}-\frac{\hbar}{2}}y_{\tau i,s}y_{i,s}y_{i,s}+\sum_{r\neq s}\frac{1+d_{i,s}^{-1}C_{i,i,s,r}d_{i,s}}{-v+x_{i,s}-\frac{\hbar}{2}}y_{\tau i,s}y_{i,s}y_{i,r}\\
    &\mspace{21mu}+\sum_{r\neq s}\frac{1}{-v+x_{i,s}-\frac{\hbar}{2}}y_{\tau i,s}y_{i,r}y_{i,r}+\sum_{r_i\neq s,r_1\neq r_2}\frac{1}{-v+x_{i,s}-\frac{\hbar}{2}}y_{\tau i,s}y_{i,r_1}y_{i,r_2}.
\end{align*}

Notice that 
\[[b_{i,0},[b_{i,0},b_{\tau(i)}(v)]]=b_{i,0}b_{i,0}b_{\tau(i)}(v)-2b_{i,0}b_{\tau(i)}(v)b_{i,0}+b_{\tau(i)}(v)b_{i,0}b_{i,0}.\]
Therefore, by Lemma \ref{lem:yserre1} and \ref{lem:yserre2}, the terms with distinct $ r_1,r_2,s $ and those with $ r_1=r_2\neq s $ have no contribution.

From
\begin{align*}
    &\mspace{24mu}4v\hbar[B_i(3v),H_i(-v)]\\
    &=\sum_{r=1}^{v_i}\frac{4v\hbar}{-3v-x_{i,r}-\frac{\hbar}{2}}[y_{i,r},H_i(-v)]\\
    &=\sum_{r=1}^{v_i}\frac{4v\hbar}{-3v-x_{i,r}-\frac{\hbar}{2}}(D_{i,i,r}(-v)-1)H_i(-v)y_{i,r}\\
    &=-4v\hbar^2H_i(-v)\sum_{r=1}^{v_i}\frac{1}{(v+x_{i,r})(v-x_{i,r}-\frac{3\hbar}{2})}y_{i,r},
\end{align*}
we know the function $4v[B_i(3v),H_i(-v)]$ has simple poles at \[\bigg\{x_{i,r}+\tfrac{3\hbar}{2}, x_{i,r}-\tfrac{\hbar}{2}, x_{i,r}+\tfrac{\hbar}{2} \mid 1\leq r\leq v_i \bigg\}.\]
One can calculate the residues at these poles and get the following expression
\begin{align*}
    &\mspace{24mu}\mathrm{Res}_{v=x_{i,s}-\frac{\hbar}{2}}(4v\hbar[B_i(3v),H_i(-v)])\\
    &=-y_{\tau i,s}y_{i,s}y_{i,s}+\sum_{r\neq s}\frac{2\hbar(2x_{i,s}-\tfrac{\hbar}{2})}{(x_{i,s}+x_{i,r}-\frac{\hbar}{2})(x_{i,s}-x_{i,r}-2\hbar)}y_{\tau i,s}y_{i,s}y_{i,r},
\end{align*}
\begin{align*}
    &\mspace{24mu}\mathrm{Res}_{v=x_{i,s}+\frac{\hbar}{2}}(4v\hbar[B_i(3v),H_i(-v)])\\
    &=2y_{i,s}y_{\tau i,s}y_{i,s}-\sum_{r\neq s}\frac{2\hbar(2x_{i,s}+\tfrac{\hbar}{2})}{(x_{i,s}+x_{i,r}+\frac{\hbar}{2})(x_{i,s}-x_{i,r}-\hbar)}y_{i,s}y_{\tau i,s}y_{i,r},
\end{align*}
\begin{equation*}
    \mathrm{Res}_{v=x_{i,s}+\frac{3\hbar}{2}}(4v\hbar[B_i(3v),H_i(-v)])=-y_{i,s}y_{i,s}y_{\tau i,s},
\end{equation*}
where the last one uses Lemma \ref{lem:Hiat3h/2}. Applying Lemma \ref{lem:partialfrac} to $4v[B_i(3v),H_i(-v)]$, we compare the coefficients in two sides.

The coefficients of $ r_1=r_2=s $ in two sides are the same:
\begin{align*}
    &\mspace{24mu}\frac{1}{-v+x_{i,s}+\frac{3\hbar}{2}}y_{i,s}y_{i,s}y_{\tau i,s}-\frac{2}{-v+x_{i,s}+\frac{\hbar}{2}}y_{i,s}y_{\tau i,s}y_{i,s}+\frac{1}{-v+x_{i,s}-\frac{\hbar}{2}}y_{\tau i,s}y_{i,s}y_{i,s}\\
    &=\frac{-1}{v-x_{i,s}+\frac{\hbar}{2}}y_{\tau i,s}y_{i,s}y_{i,s}+\frac{2}{v-x_{i,s}-\frac{\hbar}{2}}y_{i,s}y_{\tau i,s}y_{i,s}+\frac{-1}{v-x_{i,s}-\frac{3\hbar}{2}}y_{i,s}y_{i,s}y_{\tau i,s}.
\end{align*}
And the coefficients of $ r_1\neq r_2=s $ are also the same: for the coefficients of $ y_{i,r}y_{i,s}y_{\tau i,s} $,
\begin{align*}
    \frac{1+C_{i,i,r,s}-2C_{i,i,r,s}d_{i,s}C_{i,\tau i,r,s}d_{i,s}^{-1}}{-v+x_{i,s}+\frac{\hbar}{2}}=\frac{1}{-v+x_{i,s}+\frac{\hbar}{2}}\frac{2\hbar(2x_{i,s}+\frac{\hbar}{2})C_{i,i,r,s}d_{i,s}C_{i,\tau i,r,s}d_{i,s}^{-1}}{(x_{i,s}+x_{i,r}+\frac{\hbar}{2})(x_{i,s}-x_{i,r}-\hbar)},
\end{align*}
which is an identity, and for the coefficients of $ y_{\tau i,s}y_{i,s}y_{i,r} $,
\begin{equation*}
    \frac{1+d_{i,s}^{-1}C_{i,i,s,r}d_{i,s}-2C_{\tau i,i,s,r}d_{i,s}^{-1}C_{i,i,s,r}d_{i,s}}{-v+x_{i,s}-\frac{\hbar}{2}}=\frac{1}{v-x_{i,s}+\frac{\hbar}{2}}\frac{2\hbar(2x_{i,s}-\frac{\hbar}{2})}{(x_{i,s}+x_{i,r}-\frac{\hbar}{2})(x_{i,s}-x_{i,r}-2\hbar)},
\end{equation*}
which is also an identity. This finishes the proof of \eqref{equ:iSerreitaui}.

%

\bibliographystyle{alpha}
\bibliography{TruSTY.bib}

\end{document}